\documentclass[12pt]{amsart}
\usepackage[shortlabels]{enumitem}
\usepackage{amsmath, amsthm, amsfonts, amssymb, mathrsfs, graphicx, stmaryrd}
\usepackage[all]{xy}
\usepackage[usenames, dvipsnames]{color}
\usepackage[margin=1in]{geometry} 
\usepackage[bookmarks, bookmarksdepth=2, colorlinks=true, linkcolor=blue, citecolor=blue, urlcolor=blue]{hyperref}
\usepackage{eucal}
\usepackage{xcolor}

\numberwithin{equation}{section}
\newtheorem{theorem}[equation]{Theorem}

\newtheorem{proposition}[equation]{Proposition}
\newtheorem{lemma}[equation]{Lemma}

\newtheorem{question}[equation]{Question}

\newtheorem{maintheorem}{Theorem}

\theoremstyle{definition}
\newtheorem{rmk}[equation]{Remark}
\newenvironment{remark}[1][]{\begin{rmk}[#1] \pushQED{\qed}}{\popQED \end{rmk}}
\newtheorem{eg}[equation]{Example}
\newenvironment{example}[1][]{\begin{eg}[#1] \pushQED{\qed}}{\popQED \end{eg}}
\newtheorem{defnaux}[equation]{Definition}
\newenvironment{definition}[1][]{\begin{defnaux}[#1]\pushQED{\qed}}{\popQED \end{defnaux}}

\newcommand{\bA}{\mathbf{A}}

\newcommand{\cC}{\mathcal{C}}

\newcommand{\cD}{\mathcal{D}}

\newcommand{\bF}{\mathbf{F}}

\newcommand{\cL}{\mathcal{L}}

\newcommand{\bN}{\mathbf{N}}

\newcommand{\bP}{\mathbf{P}}

\newcommand{\bQ}{\mathbf{Q}}

\newcommand{\bS}{\mathbf{S}}

\newcommand{\fS}{\mathfrak{S}}

\newcommand{\fT}{\mathfrak{T}}

\newcommand{\sU}{\mathscr{U}}

\newcommand{\bZ}{\mathbf{Z}}

\newcommand{\fa}{\mathfrak{a}}

\newcommand{\rf}{\mathrm{f}}

\let\ol\overline
\let\ul\underline
\let\defn\emph
\renewcommand{\phi}{\varphi}
\renewcommand{\emptyset}{\varnothing}
\newcommand{\lw}{{\textstyle \bigwedge}}
\DeclareMathOperator{\Sym}{Sym}
\DeclareMathOperator{\Aut}{Aut}
\DeclareMathOperator{\res}{res}
\DeclareMathOperator{\Hom}{Hom}
\DeclareMathOperator{\Rep}{Rep}
\newcommand{\id}{\mathrm{id}}
\newcommand{\pol}{\mathrm{pol}}
\newcommand{\GL}{\mathbf{GL}}
\newcommand{\Th}{\mathfrak{Th}}
\newcommand{\bzero}{\mathbf{0}}

\newcommand{\Str}{\mathcal{C}}

\DeclareMathOperator{\Sh}{Sh}
\DeclareMathOperator{\CS}{\Delta}
\DeclareMathOperator{\sh}{sh}
\newcommand{\uomega}{\ul{\smash{\omega}}}
\newcommand{\ueta}{\ul{\smash{\eta}}}
\newcommand{\umu}{\ul{\smash{\mu}}}
\newcommand{\unu}{\ul{\smash{\nu}}}

\newcommand{\ulambda}{\ul{\smash{\lambda}}}

\newcommand{\DOI}[1]{\href{http://doi.org/#1}{\color{purple}{\tiny\tt DOI:#1}}}
\newcommand{\arxiv}[1]{\href{http://arxiv.org/abs/#1}{{\tiny\tt arXiv:#1}}}

\author{Nate Harman}
\address{Department of Mathematics, University of Michigan, Ann Arbor, MI}
\email{\href{mailto:nharman@umich.edu}{nharman@umich.edu}}
\urladdr{\url{https://math.mit.edu/~nharman/}}

\author{Andrew Snowden}
\address{Department of Mathematics, University of Michigan, Ann Arbor, MI}
\email{\href{mailto:asnowden@umich.edu}{asnowden@umich.edu}}
\urladdr{\url{http://www-personal.umich.edu/~asnowden/}}
\thanks{AS was supported by NSF DMS-1453893.}

\title{Ultrahomogeneous tensor spaces}
\date{August 7, 2023}

\begin{document}

\begin{abstract}
A cubic space is a vector space equipped with a symmetric trilinear form. Using categorical Fra\"iss\'e theory, we show that there is a universal ultrahomogeneous cubic space $V$ of countable infinite dimension, which is unique up to isomorphism. The automorphism group $G$ of $V$ is quite large and, in some respects, similar to the infinite orthogonal group. We show that $G$ is a linear-oligomorphic group (a class of groups we introduce), and we determine the algebraic representation theory of $G$. We also establish some model-theoretic results about $V$: it is $\omega$-categorical (in a modified sense), and has quantifier elimination (for vectors). Our results are not specific to cubic spaces, and hold for a very general class of tensor spaces; we view these spaces as linear analogs of the relational structures studied in model theory.
\end{abstract}

\maketitle
\tableofcontents

\section{Introduction}

The purpose of this paper is to study tensor spaces, i.e., vector spaces equipped with various kinds of multi-linear forms. We construct (ultra)homogeneous\footnote{In this context, ``homogeneous'' and ``ultrahomogeneous'' are often used interchangeably. We use ``ultrahomogeneous'' in the title since ``homogeneous'' has many other meanings, but we use ``homogeneous'' in the text since it is shorter.} tensor spaces and study their symmetry groups and model theory. Our work is closely related to several current topics in algebra (such as geometry of tensors and Deligne interpolation), and ties in with classical topics in model theory (such as homogeneous structures and $\omega$-categoricity). In the remainder of the introduction, we state our results in more detail and explain the motivation behind this work.

\subsection{Tensor spaces}

Fix a field $k$ of characteristic~0. Let $\ulambda=[\lambda_1, \ldots, \lambda_r]$ be a tuple of non-empty partitions. A \defn{$\ulambda$-structure} on a $k$-vector space $V$ is a tuple $\uomega=(\omega_1, \ldots, \omega_r)$ where $\omega_i \colon \bS_{\lambda_i}(V) \to k$ is a linear map, and $\bS_{\lambda}$ denotes the Schur functor associated to $\lambda$. A \defn{$\ulambda$-space} is a vector space equipped with a $\ulambda$-structure. These are the main objects of study in this paper.

Here are some examples of the above definition:
\begin{itemize}
\item Suppose $\ulambda=[(2)]$, i.e., $\ulambda$ consists of a single partition $\lambda_1=(2)$. The Schur functor $\bS_{(2)}$ is the symmetric power $\Sym^2$. Thus a $\ulambda$-structure on $V$ is just a quadratic form (or symmetric bilinear form) on $V$, and a $\ulambda$-space is a quadratic space.
\item Similarly, if $\ulambda=[(1,1)]$ then a $\ulambda$-space is a vector space equipped with an anti-symmetric bilinear form. (The Schur functor $\bS_{(1,1)}$ is the exterior square $\lw^2$.)
\item Suppose $\ulambda=[(3)]$. Then a $\ulambda$-space is a cubic space, i.e., a vector space with a cubic form (or symmetric trilinear form).
\item Finally, suppose that $\ulambda=[(2), (2)]$ consists of two copies of the partition $(2)$. Then a $\ulambda$-structure on $V$ is a pair of quadratic forms on $V$.
\end{itemize}
An \defn{embedding} of $\ulambda$-spaces $W \to V$ is an injective linear map such that the $\ulambda$-structure on $V$ pulls back to the one on $W$. An \defn{isomorphism} is an embedding that is a linear isomorphism.

The notion of $\ulambda$-space can be viewed as a linear analog of the notion of relational structure appearing in model theory; the tuple $\ulambda$ plays the role of the signature of a relational structure. This paper can be seen as an attempt to find linear analogs of some ideas related to relational structures, such as Fra\"iss\'e limits, $\omega$-categoricity, oligomorphic groups, etc.

\subsection{Homogeneous spaces}

The following is our first main theorem:

\begin{maintheorem}[Theorem~\ref{thm:homo}] \label{mainthm}
There exists a $\ulambda$-space $V_{\ulambda}$ of countable dimension with the following two properties:
\begin{enumerate}
\item Universality: if $W$ is a finite dimensional $\ulambda$-space then there is an embedding $W \to V_{\ulambda}$.
\item Homogeneity: if $W$ and $W'$ are finite dimensional subspaces of $V_{\ulambda}$, with the induced structures, then any isomorphism $W \to W'$ extends to an automorphism of $V_{\ulambda}$.
\end{enumerate}
Moreover, $V_{\ulambda}$ is unique up to isomorphism: any $\ulambda$-space of countable dimension satisfying (a) and (b) is isomorphic to $V_{\ulambda}$.
\end{maintheorem}

This theorem is rather surprising, as it distinguishes an isomorphism class of $\ulambda$-space of countable dimension; there is nothing like this in finite dimensions, except for certain small $\ulambda$, like $\ulambda=[(2)]$.

There are two ingredients in the proof of Theorem~\ref{mainthm}. The first is a categorical variant of Fra\"iss\'e's theorem. This goes back to the work of Droste--G\"obel \cite{DrosteGobel1,DrosteGobel2}, and has appeared in more recent work too \cite{Caramello,Irwin,Kubis}. We include an appendix giving a self-contained treatement. The second is a direct construction of universal (but not necessarily homogeneous) $\ulambda$-spaces. 

\begin{remark}
The notion of universal $\ulambda$-space appears implicitly in much recent work \cite{BDDE, polygeom, des, KaZ1, KaZ2}, and is closely related to the notion of strength in commutative algebra \cite{AH, ess}. In particular, \cite{KaZ2} implies that for $\ulambda=[(d)]$ and $k$ algebraically closed, a $\ulambda$-space is universal if and only if its defining form has infinite strength. The paper \cite{BDDE} proves a generalization of this result; see Remark~\ref{rmk:uni-str}. We also note that \cite{isocubic} classifies universal cubic spaces of countable dimension up to isogeny.
\end{remark}

\begin{remark}
We work in characteristic~0 throughout this paper. Our results on $\ulambda$-spaces remain valid (with the same proofs) in positive characteristic $p$, provided $p$ is larger than each $\vert \lambda_i \vert$. In fact, over a finite field of such characteristic, one can prove Theorem~\ref{mainthm} using the classical form of Fra\"iss\'e's theorem. For small $p$, our definition of $\ulambda$-space is not really correct: for example, in characteristic~2 one should allow for both quadratic forms and symmetric bilinear forms. We suspect our results could be extended to this situation, but we have not pursued it.
\end{remark}

\subsection{Symmetry groups}

Let $G_{\ulambda}$ be the automorphism group of the space $V_{\ulambda}$ in Theorem~\ref{mainthm}. This is an infinite dimensional algebraic group; implicit in this assertion is the fact that the construction of $V_{\ulambda}$ is compatible with extension of scalars. For certain small $\ulambda$, the group $G_{\ulambda}$ is recognizable; for example, $G_{[2]}$ is the split infinite orthogonal group. However, when $\ulambda$ contains a partition of size at least three, $G_{\ulambda}$ seems to be unrelated to familiar groups.

It follows from the homogeneity of $V_{\ulambda}$ that the groups $G_{\ulambda}$ are reasonably large. For example, if $\ulambda=[(d)]$ and $k$ is algebraically closed then $G_{\ulambda}$ has two orbits on the space $\bP(V_{\ulambda})$ of lines in $V_{\ulambda}$, namely, the invariant hypersurface and its complement. To more precisely quantify the size of $G_{\ulambda}$, we introduce the notion of a \emph{linear-oligomorphic group}; as the name suggests, this is the linear analog of the notion of oligomorphic group that is so important to homogeneous structures (see \cite{Cameron, Macpherson}). See \S \ref{ss:lin-olig} for the definitions. We then prove the following theorem:

\begin{maintheorem}[Theorem~\ref{thm:olig}] \label{mainthm2}
The group $G_{\ulambda}$ is linear-oligomorphic.
\end{maintheorem}

We view the $G_{\ulambda}$ as new infinite dimensional algebraic groups that are analogous to the infinite orthogonal group. There are a number of potentially fascinating directions to explore; we suggest a few here:
 
\begin{question}
How much of the theory of algebraic groups applies to $G_{\ulambda}$? Is there a Cartan subgroup, Weyl group, and Dykin diagram? What does the Lie algebra look like? Are there interesting twisted forms? As an abstract group, is it close to being simple? Is there a notion of automorphic representation?
\end{question}

We prove one additional result about the $G_{\ulambda}$: we determine its algebraic representation theory (Theorem~\ref{thm:rep}). To do this, we show that the representation theory of $G_{\ulambda}$ is equivalent to the representation theory of the ``generalized stabilizers'' studied in \cite{tcares}. The representation theory of $G_{\ulambda}$ is very similar to that of infinite classical groups, as studied in \cite{koszulcategory, penkovserganova, penkovstyrkas, infrank}.

\subsection{Model theory}

We have drawn an analogy between $\ulambda$-spaces and relational structures. Homogeneous relational structures have a number of interesting model-theoretic properties. With this in mind, we examine some of the model-theoretic properties of homogeneous $\ulambda$-spaces. The following theorem summarizes our main results:

\begin{maintheorem}[Theorems~\ref{thm:categorical} and~\ref{thm:quant}] \label{mainthm3}
The $\ulambda$-space $V_{\ulambda}$ is linearly $\omega$-categorical, and satisfies vector-quantifier elimination.
\end{maintheorem}

Briefly, ``linearly $\omega$-categorical'' means that $V_{\ulambda}$ is determined up to isomorphism by its first-order theory, and ``vector-quantifier elimination'' means that any first-order formula about $V_{\ulambda}$ is equivalent to one where there are no quantifiers over vector-valued variables (though there could still be quantifiers over scalar-valued variables). We also explain that the theory of $V_{\ulambda}$ is decidable when $k=\ol{\bQ}$. See \S \ref{s:model} for details.

\begin{remark}
The model theory of bilinear forms has been previously studied, and some results related to Theorem~\ref{mainthm3} have appeared in this case; see \cite{Granger, Kamsma}. We also note that the book \cite{Gross} contains relevant results in this case.
\end{remark}

We hope that Theorem~\ref{mainthm3} is just a first step in the model theory of $\ulambda$-spaces. Some possible next steps are raised in following question: 

\begin{question}
Does the linear analog of the Ryll-Nardzewski theorem hold? (See Remark~\ref{rmk:ryll}.) What about linear analogs of other results from model theory (e.g., omitting types)?
\end{question}

\subsection{Further results}

Fix a vector space $V$ of countable dimension. The space $\bA^{\ulambda}$ of all $\ulambda$-structures on $V$ forms a geometric object called a \emph{$\GL$-variety}. These varieties were studied in detail in \cite{polygeom}. In a follow-up paper \cite{homoten2}, we will examine how the perspective of this paper interacts with the geometry of $\bA^{\ulambda}$. One result states that that generalized orbits on $\bA^{\ulambda}$ (introduced in \cite[\S 3]{polygeom}) correspond to ``weakly homogeneous'' $\ulambda$-spaces.

\subsection{Motivation} \label{ss:deligne}

Deligne \cite{Deligne} showed that one can ``interpolate'' the representation categories $\Rep(\fS_n)$ of finite symmetric groups to obtain a novel tensor category $\ul{\Rep}(\fS_t)$, where $t$ is a complex number. (Here $\fS_t$ is simply a formal symbol.). In recent work \cite{repst}, we generalized Deligne's construction. Let $G$ be an oligomorphic group. Given a measure $\mu$ for $G$ (in a sense that we introduce), we construct a tensor category $\ul{\operatorname{Perm}}(G; \mu)$, and in some cases, an abelian envelope $\ul{\Rep}(G; \mu)$. When $G$ is the symmetric group, this recovers Deligne's $\ul{\Rep}(\fS_t)$; the parameter $t$ corresponds to a choice of measure.

Deligne (and Deligne--Milne \cite{DeligneMilne}) also showed that one can interpolate representation categories of classical groups. We expect that there is an analog of the theory developed in \cite{repst} in the algebraic case. This motivated us to look for examples of linear-oligomorphic groups, which led to the present work. In this vein, the next questions are:

\begin{question}
What is the correct notion of measure on a linear-oligomorphic group? Are there any interesting measures on the $G_{\ulambda}$'s, outside the ones corresponding to known interpolation categories?
\end{question}

In recent work  \cite{discrete}, we showed that every ``discrete'' pre-Tannakian tensor category arises from an oligomorphic group. One might hope that algebraic oligomorphic groups (a common generalization of oligomorphic and linear oligomorphic groups, see Remark~\ref{rmk:unite}), and their super analogs, explain all pre-Tannakian categories in characteristic~0.

\subsection{Notation and terminology}

We list the most important notation:
\begin{description}[align=right,labelwidth=2.5cm,leftmargin=!]
\item[ $k$ ] A field of characteristic~0.
\item[ $\ulambda$ ] A tuple of partitions $[\lambda_1, \ldots, \lambda_r]$.
\item[ $\Str_{\ulambda}$ ] The category of $\ulambda$-spaces (\S \ref{ss:spaces}).
\item[ $\bS_{\ulambda}$ ] The sum of the Schur functors of the partitions in $\ulambda$ (\S \ref{ss:poly}).
\item[ $V_{\ulambda}$ ] The universal homogeneous countable $\ulambda$-space (Definition~\ref{defn:Vlambda}).
\item[ $G_{\ulambda}$ ] The automorphism group of $V_{\ulambda}$ (Definition~\ref{defn:Glambda}).
\item[ $\CS_X(\cC)$ ] A coslice category (\S \ref{ss:coslice}).
\end{description}
We use ``finite'' for ``finite dimensional'' and ``countable'' for ``countably infinite dimensional'' in the context of vector spaces.

\subsection*{Acknowledgments}

We thank Arthur Bik and Jan Draisma for helpful discussions.

\section{Basics of tensor spaces}

\subsection{Multi-linear forms} \label{ss:forms}

Fix a field $k$ of characteristic~0. Let $V$ be a $k$-vector space and let $n$ be a non-negative integer. An \emph{$n$-form} on $V$ is a multi-linear map $V^n \to k$, or, equivalently, a linear map $V^{\otimes n} \to k$. When $n=0$, an $n$-form on $V$ is just a scalar.

Bilinear forms decompose into symmetric and skew-symmetric pieces. There is a similar, but more complicated, decomposition of $n$-forms. For a partition $\lambda$ of $n$, let $S^{\lambda}$ be the corresponding Specht module over $k$; this is the irreducible representation of the symmetric group $\fS_n$ corresponding to $\lambda$. The \emph{Schur functor} associated to $\lambda$ is defined by
\begin{displaymath}
\bS_{\lambda}(V) = \Hom_{\fS_n}(S^{\lambda}, V^{\otimes n}),
\end{displaymath}
where $\fS_n$ acts on $V^{\otimes n}$ by permuting tensor factors. A \emph{$\lambda$-form} on $V$ is a linear map $\bS_{\lambda}(V) \to k$. When $\lambda=(n)$, the Specht module is the trivial representation of $\fS_n$, and the Schur functor $\bS_{\lambda}$ is the symmetric power $\Sym^n$. In particular, a $(2)$-form is a symmetric bilinear form.

We have a canonical isomorphism
\begin{displaymath}
V^{\otimes n} = \bigoplus_{\lambda \vdash n} S^{\lambda} \otimes \bS_{\lambda}(V),
\end{displaymath}
where the sum is over all partitions of $n$. The Specht module $S^{\lambda}$ carries a basis given by the standard tableaux of shape $\lambda$. Let $\omega$ be an $n$-form on $V$. Given a partition $\lambda$ of $n$ and a standard tableau $T$ of shape $\lambda$, we obtain a $\lambda$-form $\omega_{\lambda,T}$ on $V$ as the following composition
\begin{displaymath}
\xymatrix@C=3em{
\bS_{\lambda}(V) \ar[r]^-{T \otimes \id} &
S^{\lambda} \otimes \bS_{\lambda}(V) \ar[r] &
V^{\otimes n} \ar[r]^-{\omega} & k }
\end{displaymath}
where the second map is the canonical inclusion. The construction $\omega \mapsto (\omega_{\lambda,T})$ gives a bijection between $n$-forms and tuples consisting of a $\lambda$-form for each pair $(\lambda,T)$.

For each partition $\lambda$, fix a choice $T(\lambda)$ of standard tableau of shape $\lambda$; for example, one could use the tableau whose first row is $1, \ldots, \lambda_1$, second row is $\lambda_1+1, \ldots, \lambda_1+\lambda_2$, and so on. Write $\fa(\lambda) \subset k[\fS_n]$ for the annihilator of $T(\lambda)$. Let $\omega$ be a $\lambda$-form on $V$, with $\vert \lambda \vert=n$. Define $\omega^*$ be the unique $n$-form on $V$ satisfying
\begin{displaymath}
\omega^*_{\mu,T} = \begin{cases} \omega & \text{if $\mu=\lambda$ and $T=T(\lambda)$} \\
0 & \text{otherwise} \end{cases}
\end{displaymath}
Then $\omega \mapsto \omega^*$ is a bijection between $\lambda$-forms and $n$-forms annihilated by $\fa(\lambda)$. In this way, we can encode $\lambda$-forms as $n$-forms. For example, this procedure encodes $(2)$-forms as symmetric 2-forms.

The above discussion shows that $n$-forms and $\lambda$-forms are essentially equivalent. While $n$-forms are perhaps a bit simpler, it will ultimately be more convenient for us to work with $\lambda$-forms. We therefore base our theory on them.

\subsection{Tensor spaces} \label{ss:spaces}

A \emph{tuple of partitions}, often abbreviated to just \emph{tuple}, is an ordered tuple $\ulambda=[\lambda_1, \ldots, \lambda_r]$ where each $\lambda_i$ is a partition. The $\lambda_i$ may have different sizes. We say that $\ulambda$ is \emph{pure} if each $\lambda_i$ is non-empty. This terminology was introduced in \cite{polygeom}.

The following definition introduces the main objects of study in this paper:

\begin{definition}
Let $\ulambda=[\lambda_1, \ldots, \lambda_r]$ be a tuple of partitions. A \emph{$\ulambda$-structure} on a vector space $V$ is a tuple $\uomega=(\omega_1, \ldots, \omega_r)$ where $\omega_i$ is a $\lambda_i$-form on $V$. A \emph{$\ulambda$-space} is a vector space equipped with a $\ulambda$-structure.
\end{definition}


We now introduce some additional definitions related to $\ulambda$-spaces. We say that a $\ulambda$-space is \emph{finite}, resp.\ \emph{countable}, if its dimension is finite, resp.\ countably infinite. An \defn{embedding} $W \to V$ of $\ulambda$-spaces is an injective linear map such that the given $\ulambda$-structure on $V$ pulls back to the given $\ulambda$-structure on $W$. An \defn{isomorphism} of $\ulambda$-spaces is a bijective embedding. We let $\Str_{\ulambda}$ be the category of $\ulambda$-spaces, and we let $\Str^{\rf}_{\ulambda}$ be the category of finite $\ulambda$-spaces (in both cases the morphisms are embeddings).

\begin{remark} \label{rmk:pure}
Let $\ulambda=[\lambda_1, \ldots, \lambda_r]$ be a tuple. Reindexing if necessary, suppose that $\lambda_1=\cdots=\lambda_s=\emptyset$ and $\lambda_{s+1}, \ldots, \lambda_r$ are non-empty. Let $\umu$ be the pure tuple $[\lambda_{s+1}, \ldots, \lambda_r]$. A $\ulambda$-structure on $V$ can then be identified with a pair $(\ul{c}, \uomega)$ where $\ul{c} \in k^s$ is a tuple of scalars and $\uomega$ is a $\umu$-structure on $V$. If $V \to W$ is an embedding of $\ulambda$-spaces, then the tuples of scalars for $V$ and $W$ must be equal. We thus see that $\Str_{\ulambda}$ is equivalent to a disjoint union of copies of $\Str_{\umu}$ indexed by $k^s$. This discussion shows that the ``impure'' part of $\ulambda$ is not very interesting, and for most purposes we can restrict to pure tuples. The reason we allow impure tuples is that they appear when applying the shift operation; see \S \ref{ss:shift}.
\end{remark}

We now give two examples illustrating some aspects of infinite dimensional $\ulambda$-spaces.

\begin{example}
Two $[(1)]$-spaces with non-zero forms are isomorphic if and only if they have the same dimension. Equivalently, this means that if $V$ is a vector space then the group $\GL(V)$ of all automorphisms of $V$ acts transitively on the non-zero vectors in the dual space $V^*$. Suppose that $V$ is countable, and fix a basis. Then $\GL(V)$ can be identified with the group of all column-finite invertible matrices of size $\bN \times \bN$. Similarly, $V^*$ can be identified with the space of all row vectors of size $\bN$. The transitivity of the action thus amounts to the fact that the first row of a matrix in $\GL(V)$ can be any non-zero vector. Note that the smaller group $\bigcup_{n \ge 1} \GL_n$ does not act transitively on $V^* \setminus \{0\}$.
\end{example}

\begin{example} \label{ex:quad-space}
Suppose $k$ is algebraically closed and let $(V, \omega)$ be a $[(2)]$-space. Recall that the null space of $\omega$ consists of those vectors $v$ such that $\omega(v, -)$ is identically zero. Supposing the null space of $\omega$ vanishes and $V$ has countable dimension, it is not difficult to show that $V$ has an orthonormal basis; see \cite[Chapter~2]{Gross}. In particular, we see that any two countable $[(2)]$-spaces with zero null space are isomorphic (map one orthonormal basis to the other).

The countability hypothesis above is of crucial importance. For example, over the complex numbers there are $2^{\aleph_0}$ distinct isomorphism classes of $[(2)]$-spaces of dimension $\aleph_1$ and vanishing null space; see \cite[Chapter~2]{Gross}.
\end{example}

\subsection{Shifts of polynomial functors} \label{ss:poly}

Recall that a \defn{polynomial functor} is an endofunctor of the category of vector spaces that decomposes as a direct sum of Schur functors. The category of polynomial functors is semi-simple abelian, with Schur functors as the simple objects. For a tuple $\ulambda=[\lambda_1, \ldots, \lambda_r]$, we let $\bS_{\ulambda}=\bigoplus_{i=1}^r \bS_{\lambda_i}$. Thus every finite length polynomial functor is isomorphic to $\bS_{\ulambda}$ for some tuple $\ulambda$. We refer to \cite{expos} for additional background on polynomial functors.

We define the $n$th \emph{shift} of a polynomial functor $F$, denoted $\Sh_n(F)$, to be the functor given by $(\Sh_n{F})(V)=F(k^n \oplus V)$. It is easily seen to be a polynomial functor: in fact, we have $\Sh_n(\bS_{\lambda})=\bS_{\lambda} \oplus \cdots$, where the remaining terms are Schur functors of smaller degree. One can explicitly determine $\Sh_n(\bS_{\lambda})$ in terms of Littlewood--Richardson coefficients. Given a tuple $\ulambda$, there is another tuple $\umu$, unique up to permutation, such that $\Sh_n(\bS_{\ulambda}) \cong \bS_{\umu}$. We define $\sh_n(\ulambda)$ to be the tuple $\umu$, and $\sh_n^{\circ}(\ulambda)$ to be the pure part of $\umu$ (discard all empty partitions).

\begin{example}
We have
\begin{displaymath}
\Sym^2(k^n \oplus V) = \Sym^2(k^n) \oplus (k^n \otimes V) \oplus \Sym^2(V).
\end{displaymath}
It follows that
\begin{displaymath}
\sh_n([(2)])=[{\textstyle\binom{n+1}{2}} \cdot \emptyset, n \cdot (1), (2)].
\end{displaymath}
Here $n \cdot (1)$ indicates that the partition $(1)$ appears $n$ times.\end{example}

\subsection{Shifts of tensor spaces} \label{ss:shift}

Fix a tuple $\ulambda$. Define a category $\Sh_n(\Str_{\ulambda})$ as follows. An object a tuple $(V, \ul{v}, V')$ where $V$ is a $\ulambda$-space, $\ul{v}=(v_1, \ldots, v_n)$ are linearly independent vectors in $V$, and $V'$ is a subspace of $V$ that is complementary to $\operatorname{span}(\ul{v})$. A morphism $f \colon (V, \ul{v}, V') \to (W, \ul{w}, W')$ is a morphism of $\ulambda$-spaces $f \colon V \to W$ such that $f(v_i)=w_i$ for $1 \le i \le n$ and $f(V') \subset W'$.

\begin{proposition} \label{prop:shift-str}
Let $\umu=\sh_n(\ulambda)$. If $(V, \ul{v}, V')$ is an object of $\Sh_n(\Str_{\ulambda})$ then $V'$ carries a natural $\umu$-structure, and the functor
\begin{displaymath}
\Psi \colon \Sh_n(\Str_{\ulambda}) \to \Str_{\umu}
\end{displaymath}
given by $(V, \ul{v}, V') \mapsto V'$ is an equivalence.
\end{proposition}

\begin{proof}
Suppose that $(V, \ul{v}, V')$ is an object of $\Sh_n(\Str_{\ulambda})$, and let $\uomega$ denote the $\ulambda$-structure on $V$. We regard $\uomega$ as a linear map $\bS_{\ulambda}(V) \to k$. Making the identification $\operatorname{span}(\ul{v})=k^n$, we have $V=k^n \oplus V'$, and so
\begin{displaymath}
\bS_{\ulambda}(V)=(\Sh_n{\bS_{\ulambda}})(V')=\bS_{\umu}(V').
\end{displaymath}
Thus the $\ulambda$-structure $\uomega$ on $V$ is equivalent to a $\umu$-structure $\uomega'$ on $V'$. This is how we define $\Psi$ on objects. It is clear from the definition that $\Psi$ is essentially surjective.

Suppose now that $(W, \ul{w}, W')$ is another object of $\Sh_n(\Str_{\ulambda})$. Let $\ueta$ be the $\ulambda$-structure on $W$, and let $\ueta'$ be the corresponding $\umu$-structure on $W'$. Let $f \colon V \to W$ be a linear map such that $f(v_i)=w_i$ for $1 \le i \le n$ and $f(V') \subset W'$, and let $f' \colon V' \to W'$ be the map induced by $f$. Clearly, giving $f$ is equivalent to giving $f'$. The following two diagrams are isomorphic:
\begin{displaymath}
\xymatrix@C=4em{ \bS_{\ulambda}(V) \ar[r]^-{\uomega} \ar[d]_f & k \\
\bS_{\ulambda}(W) \ar[ru]_{\ueta} }
\qquad\qquad
\xymatrix@C=4em{ \bS_{\umu}(V') \ar[r]^-{\uomega'} \ar[d]_{f'} & k \\
\bS_{\umu}(W') \ar[ru]_{\ueta'} }
\end{displaymath}
We thus see that $f$ is a map of $\ulambda$-spaces (i.e., the first diagram commutes) if and only if $f'$ is a map of $\umu$-spaces (i.e., the second diagram commutes). Thus defining $\Psi$ on morphisms by $\Psi(f)=f'$, we see that $\Psi$ is an equivalence.
\end{proof}

\begin{example}
Suppose $\ulambda=[(2)]$. An object of $\Sh_1(\Str_{\ulambda})$ is a quadratic space $(V, \omega)$ equipped with a non-zero vector $v$ and a complementary space $V'$ to the span of $v$. The above proposition states that $V'$ naturally has a $\sh_1(\ulambda)=[(2), (1), (0)]$ structure. The $(2)$-form on $V'$ is simply the restriction of $\omega$, the $(1)$-form is $x \mapsto \omega(v, x)$, and the $(0)$-form is the scalar $\omega(v,v)$.
\end{example}

We will require a variant of the above construction as well. A \emph{pinning} of a finite $\ulambda$-space is an ordered basis, and a \emph{pinned} finite $\ulambda$-space is one equipped with a pinning. Let $U$ be an $n$-dimensional pinned finite $\ulambda$-space. Define $\Sh_U(\Str_{\ulambda})$ to be the full subcategory of $\Sh_n(\Str_{\ulambda})$ spanned by objects $(V, \ul{v}, V')$ where $\operatorname{span}(\ul{v})$ is isomorphic to $U$ as a pinned $\ulambda$-space; note that such an isomorphism is unique.

\begin{proposition} \label{prop:shift-str2}
Let $U$ be an $n$-dimensional $\ulambda$-space with a pinning, and put $\unu=\sh^{\circ}_n(\ulambda)$. Then the functor $\Psi$ from Proposition~\ref{prop:shift-str} induces an equivalence $\Sh_U(\Str_{\ulambda}) \to \Str_{\unu}$.
\end{proposition}

\begin{proof}
Let $\umu=\sh_n(\ulambda)$. Recall from Remark~\ref{rmk:pure} that $\cC_{\umu}$ is a disjoint union of copies of $\cC_{\unu}$ parametrized by $k^s$ where $s$ is the number of empty partitions in $\umu$. More canonically, this $k^s$ is identified with $\bS_{\umu}(0)^*$, where~0 is the zero vector space, and so
\begin{displaymath}
\cC_{\umu} \cong \coprod_{\ueta \in \bS_{\umu}(0)^*} \cC_{\nu}.
\end{displaymath}
Given a $\umu$-space $(V, \uomega)$, the corresponding $\ueta$ is the pull-back of $\uomega$ to $0 \subset V$. In our case we have $\bS_{\umu}(-)=\bS_{\ulambda}(k^n \oplus -)$, and so $\bS_{\umu}(0)=\bS_{\ulambda}(k^n)$. Thus the copies of $\cC_{\nu}$ above parametrize the $\ulambda$-structures on $k^n$. The category $\Sh_U(\cC_{\ulambda})$ is simply the summand where $(k^n, \ueta) \cong U$ as pinned $\ulambda$-spaces.
\end{proof}

\subsection{Coslice categories} \label{ss:coslice}

Let $X$ be an object of a category $\cC$. Recall that the \defn{coslice category} $\CS_X(\cC)$ is the category of objects over $X$. Precisely, an object of $\CS_X(\cC)$ is a pair $(Y, \alpha)$ where $Y$ is an object of $\cC$ and $\alpha \colon X \to Y$ is a morphism in $\cC$. A morphism $(Y, \alpha) \to (Z, \beta)$ in $\CS_X(\cC)$ is a morphism $\gamma \colon Y \to Z$ in $\cC$ such that $\gamma \circ \alpha = \beta$.

Coslice categories will be important to us due to their appearance in Fra\"iss\'e's theorem (Theorem~\ref{A:fraisse}). They are also closely related to shift categories, as we now explain. (This is why we are interested in shift categories.) Let $U$ be a finite pinned $\ulambda$-space. If $(V, \ul{v}, V')$ is an object of $\Sh_U(\Str_{\ulambda})$ then we have a natural embedding $U \to V$ by mapping the given basis vector of $U$ to $\ul{v}$. One easily sees that this defines a functor
\begin{displaymath}
\Sh_U(\Str_{\ulambda}) \to \CS_U(\Str_{\ulambda}).
\end{displaymath}
We make one simple observation here:

\begin{proposition} \label{prop:shift-slice}
The above functor is essentially surjective.
\end{proposition}

\begin{proof}
Suppose that $(V, \alpha)$ is an object $\CS_U(\Str_{\ulambda})$, where $\alpha \colon U \to V$ is an embedding. Let $\ul{v}$ be the image of the basis of $U$ under $\alpha$, and let $V'$ be a subspace of $V$ complementary to $\alpha(U)$. Then $(V, \ul{v}, V')$ is an object of $\Sh_U(\cC)$ that maps to $(V, \alpha)$ under the above functor.
\end{proof}

\section{Homogeneous tensor spaces}

\subsection{Universal spaces}

We are primarily interested in homogeneous spaces in this paper. However, to study them we will require some results about universal spaces. We therefore examine them now. We begin by giving the formal definition:

\begin{definition}
Let $V$ be a $\ulambda$-space.
\begin{enumerate}
\item $V$ is \defn{universal} if every finite $\ulambda$-space embeds into $V$.
\item $V$ is \defn{$d$-universal} (for $d \in \bN$) if every $d$-dimensional $\ulambda$-space embeds into $V$. \qedhere
\end{enumerate}
\end{definition}

The following is our main result on these spaces:

\begin{theorem} \label{thm:univ}
Let $\ulambda$ be a pure tuple.
\begin{enumerate}
\item A universal countable $\ulambda$-space exists.
\item A $d$-universal finite $\ulambda$-space exists (for any $d \in \bN$).
\item If $V$ is a universal countable $\ulambda$-space over $k$ and $k'/k$ is any field extension then $V \otimes_k k'$ is a universal $\ulambda$-space over $k'$.
\end{enumerate}
\end{theorem}

Some related results appear in the literature; see Remark~\ref{rmk:uni-str}. We note that the purity hypothesis in the theorem is necessary by Remark~\ref{rmk:pure}. We require a few lemmas before proving the theorem. We begin with the key special case. For this, it will be convenient to use $n$-spaces (i.e., spaces equipped with $n$-forms) instead of $\ulambda$-spaces.

\begin{lemma} \label{lem:univ-1}
Let $n$ be a positive integer. Let $V$ be a vector space with basis $\{v_{i,j}\}_{1 \le i, 1 \le j \le n}$, and let $V_m$ be the span of the $v_{i,j}$ with $1 \le i \le m$ and $1 \le j \le n$. Let $x_{i,j} \in V^*$ be the dual vector to $v_{i,j}$, and consider the $n$-form
\begin{displaymath}
\omega = \sum_{i=1}^{\infty} (x_{i,1} \otimes \cdots \otimes x_{i,n}).
\end{displaymath}
Then $(V, \omega)$ is a universal $n$-space. Moreover, given $d$ there exists $m$, such that $V_m$ is a $d$-universal $n$-space.
\end{lemma}

\begin{proof}
The result is clear for $n=1$, so we assume $n>1$. Let $(W, \eta)$ be a finite dimension $n$-space, and let $y_1, \ldots, y_d$ be a basis for $W^*$. Let $(a_1, \ldots, a_n) \colon [d^n] \to [d]^n$ be a bijection; here $[d^n]$ denotes the set of integers $\{1, \ldots, d^n\}$ and $[d]^n$ denotes the set of $n$-tuples with values in $\{1, \ldots, d\}$. We thus see that $y_{a_1(i)} \otimes \cdots \otimes y_{a_n(i)}$ indexes a basis of $(W^*)^{\otimes n}$ as $i$ varies in $[d^n]$. Write
\begin{displaymath}
\eta = \sum_{i=1}^{d^n} c_i \cdot (y_{a_1(i)} \otimes \cdots \otimes y_{a_n(i)}).
\end{displaymath}
Let $m=d^n+d$, and let $\omega'$ denote the restriction of $\omega$ to $V_m$. Define a linear map $\iota^{\dag} \colon V_m^* \to W^*$ by
\begin{displaymath}
\iota^{\dag}(x_{i,j}) = \begin{cases}
c_i y_{a_1(i)} & \text{if $1 \le i \le d^n$ and $j=1$} \\
y_{a_j(i)} & \text{if $1 \le i \le d^n$ and $j>1$} \\
0 & \text{if $d^n+1 \le i \le d^n+n$ and $j=1$} \\
y_{i-d^n} & \text{if $d^n+1 \le i \le d^n+d$ and $j>1$}
\end{cases}
\end{displaymath}
and let $\iota \colon W \to V_m$ be the dual map. The first two lines above ensures that $\iota^{\dag}$ maps the $i$th term in the sum defining $\omega'$ to the $i$th term in the sum defining $\eta$ for $1 \le i \le d^n$, while the third line ensures that $\iota^{\dag}$ kills the remaining terms of $\omega'$. Thus $\iota^*(\omega')=\eta$. The fourth line in the definition of $\iota^{\dag}$ ensures that $\iota^{\dag}$ is surjective, and so $\iota$ is injective. We thus see that $\iota$ is an embedding of $n$-spaces. Hence $V_m$ is $d$-universal, and $V$ is universal.
\end{proof}

We now prove two additional lemmas that will allow us to deduce the theorem from the above lemma and some formal manipulations.

\begin{lemma} \label{lem:univ-2}
Let $\ulambda=[\lambda_1, \ldots, \lambda_r]$ be a pure tuple, let $1 \le s \le r$, and let $\ulambda'=[\lambda_1, \ldots, \lambda_s]$. Suppose that $(V, \uomega)$ is a universal (resp.\ $d$-universal) $\ulambda$-space, and put $\uomega'=(\omega_1, \ldots, \omega_s)$. Then $(V, \uomega')$ is a universal (resp.\ $d$-universal) $\ulambda'$-space.
\end{lemma}

\begin{proof}
Let $(W, \ueta')$ be a finite dimensional $\ulambda'$-space, where $\ueta'=(\eta_1, \ldots, \eta_s)$. Define $\eta_{s+1}=\cdots=\eta_r=0$, so that $\ueta=(\eta_1, \ldots, \eta_r)$ is a $\ulambda$-structure on $W$. By hypothesis, we have an embedding $(W, \ueta) \to (V, \uomega)$ of $\ulambda$-spaces. Of course, this is also defines an embedding $(W, \ueta') \to (V, \uomega')$ of $\ulambda'$-spaces, which shows that $(V, \uomega')$ is universal. The proof for $d$-universal is the same.
\end{proof}

Suppose now that $\ulambda=[\lambda_1, \ldots, \lambda_r]$ and $\umu=[\mu_1, \ldots, \mu_s]$ are two tuples. Define
\begin{displaymath}
\ulambda \cup \umu=[\lambda_1, \ldots, \lambda_r, \mu_1, \ldots, \mu_s].
\end{displaymath}
Suppose that $(V, \uomega)$ is a $\ulambda$-space and $(V', \uomega')$ is a $\umu$-space. We regard $\uomega_i$ as a $\lambda_i$-structure on $V \oplus V'$ by pull-back along the projection $V \oplus V' \to V$, and similarly for $\uomega_j'$. In this way,
\begin{displaymath}
(\omega_1, \ldots, \omega_r, \omega'_1, \ldots, \omega'_s)
\end{displaymath}
is a $(\ulambda \cup \umu)$-structure on $V \oplus V'$.

\begin{lemma} \label{lem:univ-3}
Maintain the above notation. Suppose that $V$ is a universal (resp.\ $d$-universal) $\ulambda$-space and $V'$ is a universal (resp.\ $d$-universal) $\umu$-space. Then $V \oplus V'$ is a universal (resp.\ $d$-universal) $(\ulambda \cup \umu)$-space.
\end{lemma}

\begin{proof}
Let $(W, \ueta)$ be a finite dimensional $(\ulambda \cup \umu)$-space, and write $\ueta=(\eta_1, \ldots, \eta_r, \eta'_1, \ldots, \eta'_s)$. By hypothesis, there are embeddings
\begin{displaymath}
\phi \colon (W, \eta_1, \ldots, \eta_r) \to (V, \uomega), \qquad
\psi \colon (W, \eta'_1, \ldots, \eta'_s) \to (V', \uomega')
\end{displaymath}
of $\ulambda$- and $\umu$-spaces. One easily verifies that $\phi \oplus \psi \colon W \to V \oplus V'$ is an embedding of $(\ulambda \cup \umu)$-spaces, and so $V \oplus V'$ is universal. The proof for $d$-universal is the same.
\end{proof}

\begin{proof}[Proof of Theorem~\ref{thm:univ}]
(a) Let $\ulambda=[\lambda_1, \ldots, \lambda_r]$ be a pure tuple. Let $n$ be the size of $\lambda_1$. Let $\umu$ be the tuple such that $\bS_{\umu}$ is the $n$th tensor power functor; thus a $\umu$-structure is equivalent to giving an $n$-form. By Lemma~\ref{lem:univ-1} there is a universal countable $\umu$-space. Since $\lambda_1$ is one of the $\mu_i$'s, it follows from Lemma~\ref{lem:univ-2} that there is a universal countable $\lambda_1$-space. Similarly for the other $\lambda_i$'s. Lemma~\ref{lem:univ-3} now shows that there is a universal countable $\ulambda$-space.

(b) This follows just like (a).

(c) Let $\Phi \colon \Str^{\rf}_{\ulambda,k} \to \Str^{\rf}_{\ulambda,k'}$ be the base change functor, and identify countable $\ulambda$-spaces with ind-objects in $\Str^{\rf}_{\ulambda}$ (see \S \ref{ss:ind}). By Proposition~\ref{A:Phi-univ}, it suffices to show that $\Phi$ carries some universal ind-object of $\Str^{\rf}_{\ulambda,k}$ to a universal ind-object of $\Str^{\rf}_{\ulambda,k'}$. Let $V$ (resp.\ $V'$) be the universal countable $\ulambda$-space over $k$ (resp.\ $k'$) constructed in (a). By inspection, $V'=\Phi(V)$; the point is simply that the field did not figure into the construction at all. This completes the proof.
\end{proof}

\begin{remark} \label{rmk:uni-str}
Suppose $k$ is an algebraically closed field and $\ulambda$ is a pure tuple. Let $V$ be a $k$-vector space of countable dimension, and let $\bA^{\ulambda}$ be the infinite dimensional affine scheme parametrizing $\ulambda$-structures on $V$ (see \cite[\S 1.1]{polygeom}). The general linear group $\GL(V)$ acts on $\bA^{\ulambda}$, and a point is called \defn{$\GL$-generic} if it has dense $\GL(V)$-orbit \cite[\S 3.4]{polygeom}. By \cite[Proposition~3.13]{polygeom}, $\GL$-generic points exist; the proof is essentially the same as the proof of Theorem~\ref{thm:univ}(a) given above. An important result \cite[Corollary~2.6.3]{BDDE} implies that $(V, \uomega)$ is universal if and only if $\uomega$ is $\GL$-generic. (We note that this was first proved in the case symmetric powers by \cite{KaZ2}.)
\end{remark}

\subsection{Homogeneous spaces}

We now reach the central concept of this paper:

\begin{definition}
A $\ulambda$-space $V$ is \emph{homogeneous} if for any finite dimensional subspaces $W$ and $W'$ of $V$, any isomorphism $W \to W'$ of $\ulambda$-spaces extends to an automorphism of $V$.
\end{definition}

The following is the first main theorem of this paper:

\begin{theorem} \label{thm:homo}
Let $\ulambda$ be a pure tuple.
\begin{enumerate}
\item A universal homogeneous countable $\ulambda$-space exists.
\item Any two universal homogeneous countable $\ulambda$-spaces are isomorphic.
\item If $V$ is a universal homogeneous countable $\ulambda$-space over $k$ and $k'/k$ is any field extension then $V \otimes_k k'$ is a universal homogeneous $\ulambda$-space over $k'$.
\end{enumerate}
\end{theorem}

\begin{proof}
Throughout this proof, we identify countable $\ulambda$-spaces with ind-objects in the category $\Str^{\rf}_{\ulambda}$; see \S \ref{ss:ind}. We also note that, since $\ulambda$ is pure, $\cC_{\ulambda}$ has an initial object, namely, the zero space. Thus $\cC_{\ulambda}$ is the coslice category $\Delta_U(\cC_{\ulambda})$ with $U=0$.

(a) By Fra\"iss\'e's theorem (Theorem~\ref{A:fraisse}), it suffices to show that each coslice category $\Delta_U(\cC_{\ulambda}^{\rf})$ has a universal ind-object. Let $U$ be an $n$-dimensional $\ulambda$-space, and choose a pinning (basis) of $U$. We have an equivalence $\Sh_U(\Str_{\ulambda}) \cong \Str_{\umu}$, where $\umu=\sh_n^{\circ}(\ulambda)$, by Proposition~\ref{prop:shift-str2}; this clearly induces an equivalence $\Sh_U(\Str^{\rf}_{\ulambda}) \cong \Str^{\rf}_{\umu}$. By Theorem~\ref{thm:univ}, the category $\Str_{\umu}^{\rf}$ has a universal ind-object. We thus see that the same is true for $\Sh_U(\Str_{\ulambda}^{\rf})$. The forgetful functor $\Sh_U(\Str_{\ulambda}^{\rf}) \to \CS_U(\Str_{\ulambda}^{\rf})$ is essentially surjective (Proposition~\ref{prop:shift-slice}). Thus $\CS_U(\Str_{\ulambda}^{\rf})$ also has a universal ind-object (Proposition~\ref{A:Phi-univ}).

(b) Universal homogeneous ind-objects are always unique up to isomorphism; see Proposition~\ref{prop:finj}(b).

(c) Let $\Phi \colon \Str^{\rf}_{\ulambda,k} \to \Str^{\rf}_{\ulambda,k'}$ be the base change functor $V \mapsto V \otimes_k k'$. To prove the statement, it suffices to show that $\Phi$ preserves universal ind-objects on coslice categories (Proposition~\ref{A:Phi-homo}).

Let $U$ be an $n$-dimensional $\ulambda$-space over $k$, and choose a pinning on $U$. Let $U'=U \otimes_k k'$. We have a commutative (up to isomorphism) diagram
\begin{displaymath}
\xymatrix@C=3em{
\Str^{\rf}_{\umu,k} \ar@{=}[r] \ar[d] &
\Sh_U(\Str^{\rf}_{\ulambda,k}) \ar[r] \ar[d] &
\CS_U(\Str^{\rf}_{\ulambda,k}) \ar[d] \\
\Str^{\rf}_{\umu,k'} \ar@{=}[r] &
\Sh_{U'}(\Str^{\rf}_{\ulambda,k'}) \ar[r] &
\CS_{U'}(\Str^{\rf}_{\ulambda,k'}) }
\end{displaymath}
On the top line the first equivalence comes from Proposition~\ref{prop:shift-str2}, and the second functor is the forgetful functor, which is essentially surjective (Proposition~\ref{prop:shift-slice}); similarly on the second line. The vertical maps are base change maps.

Let $\Omega_1$ be a universal ind-object in $\Sh_U(\Str^{\rf}_{\ulambda,k})$ and let $\Omega_2$, $\Omega_3$, and $\Omega_4$ be its images in the categories $\CS_U(\Str^{\rf}_{\ulambda,k})$, $\Sh_{U'}(\Str^{\rf}_{\ulambda,k'})$, and $\CS_{U'}(\Str^{\rf}_{\ulambda,k'})$. The left vertical functor preserves universal ind-objects by Theorem~\ref{thm:univ}(b), and so the middle one does as well; thus $\Omega_3$ is universal. Since the final functors on each line are essentially surjective, they preserve universal ind-objects (Proposition~\ref{A:Phi-univ}); thus $\Omega_2$ and $\Omega_4$ are universal. We thus see that the right functor maps the universal ind-object $\Omega_2$ to the universal ind-object $\Omega_4$ (up to isomorphism), and so this functor preserves all universal ind-objects (Proposition~\ref{A:Phi-univ}).
\end{proof}

\begin{definition} \label{defn:Vlambda}
For a pure tuple $\ulambda$, we let $V_{\ulambda}$ denote the universal homogeneous countable $\ulambda$-space. It is well-defined up to isomorphism.
\end{definition}

\begin{example} \label{ex:homo-quad}
Suppose that $\ulambda=[(2)]$, so that a $\ulambda$-space is a quadratic space. Recall that the \defn{hyperbolic plane} $H$ is the 2-dimensional quadratic space with form $xy$. We claim that $V_{\ulambda}$ can be taken to be $H^{\oplus \infty}$ (countable orthogonal sum of copies of $H$).

We first show that $H^{\oplus \infty}$ is universal. Let $W$ be a finite dimensional quadratic space. Decompose $W$ into an orthogonal direct sum $W_0 \oplus W_1$ where $W_1$ is the null space of the form and $W_0$ is non-degenerate. It is well-known that $W_0 \oplus W_0' \cong H^{\oplus n}$, where $n=\dim(W_0)$ and $W_0'$ is obtained from $W_0$ by negating the form. Thus $W_0$ embeds into $H^{\oplus n}$. Now, suppose that $W_1$ is $m$-dimensional with basis $e_1, \ldots, e_m$. Then $W_1$ embeds into $H^{\oplus m}$ by mapping $e_i$ to an isotropic vector in the $i$th summand. We thus see that $W$ embeds into $H^{\oplus (n+m)}$, and thus into $H^{\oplus \infty}$. Hence $H^{\oplus \infty}$ is universal. (In fact, this argument shows that $H^{\oplus d}$ is $d$-universal.)

Now suppose that $\phi \colon W \to W'$ is an isometry of finite dimensional subspaces of $H^{\oplus \infty}$. Let $n$ be such that $W$ and $W'$ are contained in $H^{\oplus n}$. By Witt's theorem, $\phi$ extends to an isometry of $H^{\oplus n}$, which in turn extends to an isometry of $H^{\oplus \infty}$. Thus $H^{\oplus \infty}$ is homogeneous. (In the language of this paper, Witt's theorem simply states that a finite dimensional non-degenerate quadratic space is homogeneous.)

We note that $V_{\ulambda}$ admits other descriptions too. For instance, if $-1$ is a square in $k$ then one can take $V_{\ulambda}$ to be the quadratic space with form $\sum_{i \ge 1} x_i^2$.
\end{example}

\begin{example} \label{ex:22}
Suppose $k$ is algebraically closed and $\ulambda=[(2),(2)]$. Let $V=k^{\oplus \infty}$, and represent a $\ulambda$-structure on $V$ by a pair $(A,B)$ of infinite symmetric matrices. Note that all entries in $A$ or $B$ are allowed to be non-zero. One can show that $(A,B)$ is a universal homogeneous structure if and only if the columns of $A$ and $B$, taken all together, are linearly independent.
\end{example}

\begin{remark}
Every countable $\ulambda$-space embeds into $V_{\ulambda}$ by Proposition~\ref{prop:finj}(c). Thus (the $\ulambda$-structure on) $V_{\ulambda}$ belongs to the maximal class in \cite[Theorem~2.9.1]{BDDE}.
\end{remark}

\subsection{Classification} \label{ss:classification}

When considering some class of relational structures, a common problem is to classify the homogeneous structures. For example, \cite{LachlanWoodrow} classified homogeneous graphs. With this in mind, a natural problem is to classify homogeneous (but not necessarily universal) $\ulambda$-spaces.

To begin, we note that non-universal homogeneous $\ulambda$-spaces do exist. Here are some examples:
\begin{itemize}
\item Any finite dimensional non-degenerate quadratic space is homogeneous by Witt's theorem, and obviously not universal.
\item Over the real numbers, the countable quadratic space with form $\sum_{i \ge 1} x_i^2$ is homogeneous (again by Witt's theorem), but not universal (since it is positive definite).
\item Any space equipped with the zero $\ulambda$-structure is homogeneous and not universal.
\item More generally, if $V$ is a homogeneous $\lambda$-space then it is also a homogeneous $[\lambda,\mu]$-space when given the zero $\mu$-structure.
\end{itemize}
From now on, we work over an algebraically closed field and confine our attention to infinite dimensional spaces. It seems plausible that in this case all homogeneous spaces come from taking a universal homogeneous space and padding by zeros, as in the final example above. We have only proved this in the following special case:

\begin{proposition} \label{prop:homo-class}
Suppose $k$ is algebraically closed and $\ulambda=[(d)]$ for some positive integer $d$. Let $V$ be a homogeneous countable $\ulambda$-space with non-zero form $f$. Then $V$ is universal.
\end{proposition}

\begin{proof}
Suppose $f$ is degenerate in the sense of \cite[\S 9.1]{polygeom}. Write $f=\Phi(g_1, \ldots, g_r)$ where $\Phi$ is a polynomial, $g_i \in (\Sym^{d_i}{V})^*$, and $(g_1, \ldots, g_r)$ is non-degenerate, which is possible by iteratively applying \cite[Proposition~9.1]{polygeom}, and take $r$ minimal. Re-ordering if necessary, suppose $d_1 \ge d_2 \ge \cdots \ge d_r$, and let $s$ be such that $d_1=\cdots=d_s$ and $d_s \ne d_{s+1}$. If $\sigma$ is an automorphism of $V$ then $\sigma$ preserves the $k$-span of $g_1, \ldots, g_s$ by \cite[Theorem~9.5]{polygeom}.

Now, for any $c_1, \ldots, c_r \in k$ we can find a non-zero $v \in V$ such that $g_i(v)=c_i$ for $1 \le i \le r$. This follows from \cite[Corollary~2.6.3]{BDDE} (see also \cite[\S 2]{isocubic}). Let $v$ be a non-zero vector such that $g_i(v)=0$ for all $i$. Let $c_1, \ldots, c_r \in k$ be such that $\Phi(c_1, \ldots, c_r)=0$ but $c_1 \ne 0$, and let $w \in V$ be such that $g_i(w)=c_i$ for all $1 \le i \le r$. Thus $f(v)=f(w)=0$, and so $v$ and $w$ generate isomorphic 1-dimensional $\ulambda$-subspaces of $V$. However, there is no automorphism of $V$ moving $v$ to $w$, since $\Aut(V)$ preserves the locus $g_1=\cdots=g_s=0$. Thus $V$ is not homogeneous, a contradiction.

The above argument shows that $f$ is non-degenerate, and so $V$ is universal by Remark~\ref{rmk:uni-str}.
\end{proof}

\begin{remark}
When $d=2$ the above proof shows that a quadratic space of infinite dimension but finite non-zero rank is not homogeneous. The key point is that there are two types of isotropic vectors: those in the null space and those not in the null space.
\end{remark}


\section{Linear-oligomorphic groups}

\subsection{Linear-oligomorphic groups} \label{ss:lin-olig}

Recall that an \emph{oligomorphic group} is a group $G$ with a faithful action on a set $\Omega$ such that $G$ has finitely many orbits on $\Omega^n$ for all $n \ge 0$. The most basic example is the infinite symmetric group, i.e., take $\Omega$ to be an infinite set and $G$ the group of all permutations of $\Omega$. There is an intimate connection between oligomorphic groups and homogeneous structures; see \cite{Cameron} or \cite{Macpherson} for general background.

We aim to extend this story to the linear case. To this end, we introduce the following concept:

\begin{definition} \label{defn:olig}
Let $V$ be a vector space and let $G$ be a subgroup of $\GL(V)$. We say that $G$ is \emph{linear-oligomorphic} if for any $d \ge 0$ there exists a finite dimensional subspace $E$ of $V$ such that if $W$ is an $d$-dimensional subspace of $V$ then there exists $g \in G$ such that $gW \subset E$.
\end{definition}

Note that a permutation group $(G, \Omega)$ is oligomorphic if and only if the following condition holds: for every $n \ge 0$ there is a finite subset $S$ of $\Omega$ such that if $T$ is any $n$-element subset of $\Omega$ then there is some $g \in G$ such that $gT \subset S$. The above definition simply reformulates this in linear terms.

\begin{example}
The group $\GL(V)$ is clearly linear-oligomorphic. It is also easy to see directly that the split infinite orthogonal group and the infinite symplectic group are linear-oligomorphic.
\end{example}

\begin{remark}
Definition~\ref{defn:olig} is not intended to be definitive: e.g., one might also want $G$ to be a closed ind-subscheme of $\GL(V)$, and for the defining property to hold after passing to extensions of $k$. However, for the purposes of this paper, Definition~\ref{defn:olig} will suffice.
\end{remark}

\begin{remark} \label{rmk:unite}
Consider an oligomorphic group $G$ acting on the set $\Omega$. We can linearize $\Omega$ to obtain a permutation representation $V=k[\Omega]$ of $G$. It is not necessarily true that $G$ is a linear-oligomorphic subgroup of $\GL(V)$. For instance, if $G$ is the symmetric group on the infinite set $\Omega=\bN$ then $G$ is not linear-oligomorphic: indeed, if $L_m$ is the 1-dimensional subspace of $k[\Omega]$ spanned by $\sum_{i=1}^m e_i$ then any finite dimensional subspace of $V$ contains conjugates of only finitely many of the $L_m$'s.

We would like to have a class of groups (``algebraic-oligomorphic groups'') that includes both oligomorphic and linear-oligomorphic groups. One reason for this is for applications to tensor categories, as described in \S \ref{ss:deligne}. Here is a candidate. Consider a group ind-scheme $G$ equipped with a class $\sU$ of closed subgroups such that the following conditions hold:
\begin{enumerate}
\item $\sU$ is closed under conjugation and finite intersections.
\item We have $\bigcap_{U \in \sU} U=1$.
\item Given $U,V \in \sU$ there is a closed subscheme $X$ of $G$ that is of finite type over $k$ such that the map $U \times X \times V \to G$ given by $(u,x,v) \mapsto uxv$ is surjective. \qedhere
\end{enumerate}
An oligomorphic group $G$ satisfies the above definition: regard $G$ as a discrete ind-scheme, and take $\sU$ to be the groups defining the admissible topology discussed in \cite[\S 2.2]{repst}. Now let $G \subset \GL(V)$ be linear-oligomorphic, and suppose that $G$ is Zariski closed (and thus a group ind-scheme). Then $G$ satisfies the above conditions, by taking $\sU$ to be the subgroups that pointwise fix some finite dimensional subspace of $V$.
\end{remark}

\subsection{Automorphism groups}

Fix a pure tuple $\ulambda$. We now introduce the following important object:

\begin{definition} \label{defn:Glambda}
We let $G_{\ulambda}$ be the automorphism group of the $\ulambda$-space $V_{\ulambda}$.
\end{definition}

\begin{example}
We give some examples of the $G_{\ulambda}$ groups:
\begin{enumerate}
\item $G_{\emptyset}$ is the infinite general linear group, where $\emptyset$ is the empty tuple.
\item $G_{[(1)]}$ is the stabilizer in $\GL(V_{[1]})$ of a non-zero linear functional.
\item $G_{[(2)]}$ is the split infinite orthogonal group (see Example~\ref{ex:homo-quad}).
\item $G_{[(1,1)]}$ is the infinite symplectic group.
\item $G_{[(2),(2)]}$ is the intersection of two generic conjugates of the infinite orthogonal group (the stabilizers of $A$ and $B$ in Example~\ref{ex:22}, if $k$ is algebraically closed).
\item $G_{[(3)]}$ does not seem to be closely related to any familiar groups. \qedhere
\end{enumerate}
\end{example}

The following theorem shows that the group $G_{\ulambda}$ is reasonably large:

\begin{theorem} \label{thm:olig}
The group $G_{\ulambda}$ is linear-oligomorphic.
\end{theorem}

\begin{proof}
Let $d \ge 0$ be given. Let $E$ be a $d$-universal finite $\ulambda$-space, which exists by Theorem~\ref{thm:univ}. Fix an embedding $E \to V_{\ulambda}$, which exists because $V_{\ulambda}$ is universal, and treat $E$ as a subspace of $V_{\ulambda}$. Now suppose that $W$ is an $d$-dimensional subspace of $V_{\ulambda}$, and let $i \colon W \to V_{\ulambda}$ denote the inclusion. Since $E$ is $d$-universal, there is an embedding $j \colon W \to E$. Since $V_{\ulambda}$ is homogeneous, there is $g \in G_{\ulambda}$ such that $g \circ i = j$. It follows that $gW \subset E$, and so $G_{\ulambda}$ is linear-oligomorphic.
\end{proof}

Given a group $G$ acting on a set $X$ and a tuple $\ul{x}=(x_1, \ldots, x_m)$ in $X^m$, we let $G(\ul{x})$ be the subgroup of $G$ consisting of elements that fix each $x_i$. We will require the following strengthening of the above theorem when we prove quantifier elimination in \S \ref{ss:quant}.

\begin{proposition} \label{prop:olig2}
For any $\ul{u} \in V_{\ulambda}^m$ the group $G_{\ulambda}(\ul{u})$ is linear-oligomorphic.
\end{proposition}

\begin{proof}
Fix $d \ge 0$. Let $U$ be the span of $u_1, \ldots, u_m$, regarded as a pinned $\ulambda$-space. Recall (Proposition~\ref{prop:shift-str2}) that we have an equivalence $\Sh_U(\cC_{\ulambda})=\cC_{\umu}$, where $\umu=\sh_m^{\circ}(\ulambda)$, given by $(V, \ul{v}, V') \mapsto V'$. Let $E'$ be an $d$-universal finite $\umu$-space, which exists by Theorem~\ref{thm:univ}, and let $(E, \ul{e}, E')$ be the corresponding object of $\Sh_U(\cC_{\ulambda})$. Fix an embedding $E \to V_{\ulambda}$ of $\ulambda$-spaces, and identify $E$ with a subspace of $V_{\ulambda}$ in what follows. Since $V_{\ulambda}$ is homogeneous, there is $g \in G_{\ulambda}$ such that $g \ul{e} = \ul{u}$. Replacing $E$ with $gE$, we assume $\ul{e}=\ul{u}$.

Now let $W'$ be a $d$-dimensional subspace of $V_{\ulambda}$ such that $U \cap W'=0$, and let $W=U+W'$. We let $i \colon W \to V_{\ulambda}$ be the inclusion, and we regard $(W, \ul{u}, W')$ as an object of $\Sh_U(\cC_{\ulambda})$. Since $E'$ is $d$-universal, there is an embedding $W' \to E'$ of $\umu$-spaces, which translates to an embedding $(W, \ul{u}, W') \to (E, \ul{u}, E')$ in $\Sh_U(\cC_{\ulambda})$. In other words, we have an embedding $j \colon W \to V_{\ulambda}$ such that $j(\ul{u})=\ul{u}$ and $j(W) \subset E$. Since $V_{\ulambda}$ is homogeneous, there is $g \in G_{\ulambda}$ such that $j=g \circ i$. This shows that there is $g \in G_{\ulambda}(\ul{u})$ such that $g W' \in E'$.

Changing notation slightly, we have thus proved the following: there is a finite dimensional subspace $E$ of $V_{\ulambda}$ such that if $W$ is a $d$-dimensional subspace of $V_{\ulambda}$ with $W \cap U=0$ then there is $g \in G_{\ulambda}(\ul{u})$ such that $gW \subset E$. Of course, we are free to enlarge $E$, and so we may as well suppose $U \subset E$. Suppose now that $W$ is an arbitrary $d$-dimensional subspace of $V_{\ulambda}$. We have $W \subset U+W'$ for some $d$-dimensional subspace $W'$ with $U \cap W'=0$. Letting $g \in G_{\ulambda}(\ul{u})$ be such that $gW' \subset E$, we have $gW \subset E$. Thus $G_{\ulambda}(\ul{u})$ is linear-oligomorphic.
\end{proof}

\begin{remark}
There are other examples of linear-oligomorphic groups. For instance, if $k$ is algebraically closed then the automorphism group of any countable quadratic space is linear-oligomorphic. More generally, we show in \cite{homoten2} that the automorphism group of any weakly homogeneous $\ulambda$-space is linear-oligomorphic.
\end{remark}

\subsection{Representation theory}

Let $G$ be a subgroup of $\GL(V)$. We can then regard $V$ as a representation of $G$; we call it the \emph{standard representation}. We say that a representation of $G$ is \emph{polynomial} if it occurs as a subquotient of a (possibly infinite) direct sum of tensor powers of the standard representation. We let $\Rep^{\pol}(G)$ be the category of polynomial representation of $G$. It is a Grothendieck abelian category and is closed under tensor product. We believe that studying this category for linear-oligomorphic groups should be an interesting problem.

In fact, the paper \cite{tcares} essentially solves this problem when $G=G_{\ulambda}$. The following theorem describes the most important points. This generalizes results from \cite{koszulcategory, penkovserganova, penkovstyrkas, infrank} on infinite rank classical groups.

\begin{theorem} \label{thm:rep}
Let $\ulambda$ be a pure tuple. In what follows, we consider representations of $G_{\ulambda}$.
\begin{enumerate}
\item The representation $V_{\ulambda}^{\otimes n}$ is of finite length, for any $n \ge 0$.
\item For a partition $\mu$, the socle $L_{\mu}$ of $\bS_{\mu}(V_{\ulambda})$ is irreducible; every irreducible polynomial representation is isomorphic to $L_{\mu}$ for a unique $\mu$.
\item The representation $\bS_{\mu}(V_{\ulambda})$ is injective in the category $\Rep^{\pol}(G_{\ulambda})$.
\item Every finite length polynomial representation has finite injective dimension.
\end{enumerate}
\end{theorem}

Before proving the theorem, we recall some material from \cite{tcares}. Let $\ul{\omega}$ be a $\ulambda$-structure on a countable vector space $V$. Write $V=\bigcup_{n \ge 1} V_n$ with each $V_n$ finite. Define $\Gamma_{\ul{\omega}}(n)$ to be the set of all $g \in \GL(V)$ such that $g^{-1} \ul{\omega}$ and $\ul{\omega}$ have the same restriction to $V_n$. This is typically just a set, and not a subgroup. The \defn{generalized stabilizer} of $\ul{\omega}$ is the system $\Gamma_{\ul{\omega}}=\{\Gamma_{\ul{\omega}}(n)\}_{n \ge 1}$. One of the main ideas of \cite{tcares} is that, while the stabilizer of $\ul{\omega}$ is often ``too small,'' the generalized stabilizer is always large enough.

A \emph{pre-representation} of $\Gamma_{\ul{\omega}}$ is a vector space $W$ equipped with a partially defined action map $\Gamma_{\ul{\omega}} \times W \dashrightarrow W$. A little more precisely, for each $w \in W$ there must exist $n \ge 1$ such that the action of $g$ on $w$ is defined for all $g \in \Gamma_{\ul{\omega}}(n)$. A \emph{representation} of $\Gamma_{\ul{\omega}}$ is a pre-representation satisfying some natural conditions. We refer to \cite[\S 7.2]{tcares} for the exact definitions. Since each $\Gamma_{\ul{\omega}}(n)$ is contained in $\GL(V)$, there is a natural pre-representation of $\Gamma_{\ul{\omega}}$ on $V$, which is a representation. A representation of $\Gamma_{\ul{\omega}}$ is \emph{polynomial} if it occurs as a subquotient of a direct sum of tensor powers of $V$. We let $\Rep^{\pol}(\Gamma_{\ul{\omega}})$ denote the category of polynomial representations.

Let $G_{\ul{\omega}}$ be the stabilizer of $\ul{\omega}$ in $\GL(V)$. We have $G_{\ul{\omega}} \subset \Gamma_{\ul{\omega}}(n)$ for all $n \ge 1$. In particular, if $W$ is a representation of $\Gamma_{\ul{\omega}}$ then there is well-defined action map $G_{\ul{\omega}} \times W \to W$, which is easily seen to define a representation of $G_{\ul{\omega}}$ on $W$; we call this the \defn{restriction} of $W$ to $G_{\ul{\omega}}$. It is clear that if $W$ is a polynomial representation of $\Gamma_{\ul{\omega}}$ then its restriction is a polynomial representation of $G_{\ul{\omega}}$. We thus have a restriction functor
\begin{equation} \label{eq:res}
\res \colon \Rep^{\pol}(\Gamma_{\ul{\omega}}) \to \Rep^{\pol}(G_{\ul{\omega}})
\end{equation}
The following two lemmas are the key results needed for Theorem~\ref{thm:rep}:

\begin{lemma} \label{lem:rep1}
Suppose $(V, \ul{\omega})$ is homogeneous. Then $G_{\ul{\omega}}$ is dense in $\Gamma_{\ul{\omega}}$ in the following sense: given $g \in \Gamma_{\ul{\omega}}(n)$ there exists $h \in G_{\ul{\omega}}$ such that $g$ and $h$ have the same restriction to $V_n$.
\end{lemma}

\begin{proof}
Let $g$ be as in the statement of the lemma. Then $g \colon (V_n, \ul{\omega}) \to (gV_n, \ul{\omega})$ is an isomorphism of $\ulambda$-spaces. By homogeneity, there is $h \in G_{\ul{\omega}}$ such that $g$ and $h$ have equal restriction to $V_n$.
\end{proof}

\begin{lemma} \label{lem:rep2}
Suppose $(V, \ul{\omega})$ is homogeneous. Then \eqref{eq:res} is an equivalence.
\end{lemma}

\begin{proof}
It is clear that $\res$ is faithful. We now show that it is full. Thus let $W$ and $W'$ be polynomial representations of $\Gamma_{\ul{\omega}}$ and let $f \colon W \to W'$ be a $G_{\ul{\omega}}$-linear map. Write $W=W_1/W_2$ where $W_2 \subset W_1$ are $\Gamma_{\ul{\omega}}$-subrepresentations of a sum of tensor powers of $V$, and similarly write $W'=W'_1/W'_2$. Let $w \in W$ be given, let $w'=f(w')$, and let $w_1 \in W_1$ and $w'_1 \in W'_1$ be lifts. Let $n$ be such that $w_1$ and $w'_1$ belong to the appropriate sums of tensor powers of $V_n$. We claim that $f(gw)=gf(w)$ for all $g \in \Gamma_{\ul{\omega}}(n)$, which will show that $f$ is a map of $\Gamma_{\ul{\omega}}$-representations. Thus let $g$ be given. Appealing to Lemma~\ref{lem:rep1}, let $h \in G_{\ul{\omega}}$ have the same restriction to $V_n$ as $g$. Then $gw_1=hw_1$ and $gw'_1=hw'_1$. Thus $f(gw)=f(hw)=hf(w)=gf(w)$, where in the second step we used that $f$ is $G_{\ul{\omega}}$-equivariant. This proves the claim, and the fullness of $\res$ follows.

It remains to show that $\res$ is essentially surjective. Thus let $W$ be a polynomial representation of $G_{\ul{\omega}}$. Write $W=W_1/W_2$ where $W_2 \subset W_1$ are $G_{\ul{\omega}}$-subrepresentations of a sum of tensor powers of $V$. We claim that $W_1$ and $W_2$ are $\Gamma_{\ul{\omega}}$-subrepresentations; this will imply that $W$ is naturally a polynomial $\Gamma_{\ul{\omega}}$-representation, and establish essential surjectivity. Let $w \in W_1$ be given and let $n$ be such that $w$ belongs to the appropriate sum of tensor powers of $V_n$. Given $g \in \Gamma_{\ul{\omega}}(n)$, by Lemma~\ref{lem:rep1} there is $h \in G_{\ul{\omega}}$ with the same restriction to $V_n$; thus $gw=hw$ belongs to $W_1$. Thus $gw \in W_1$ for all $g \in \Gamma_{\ul{\omega}}(n)$, which proves that $W_1$ is a $\Gamma_{\ul{\omega}}$-subrepresentation. The proof for $W_2$ is similar.
\end{proof}

Theorem~\ref{thm:rep} follows from Lemma~\ref{lem:rep2} and properties of $\Rep^{\pol}(\Gamma_{\ul{\omega}})$ established in \cite{tcares}.

\begin{remark}
The category $\Rep^{\pol}(\Gamma_{\uomega})$ is equivalent to several other categories, as discussed in \cite[\S 1.5]{tcares}. Thus $\Rep^{\pol}(G_{\ulambda})$ is equivalent to these categories as well. For instance, $\Rep^{\pol}(G_{\ulambda})$ is equivalent to the category of locally finite representations of the upwards $\ulambda$-Brauer category, which gives a combinatorial description of the category.
\end{remark}

\begin{question}
The paper \cite{spinrep} develops the theory of the spin representation for the infinite orthogonal group. Is there an analogous theory for $G_{\ulambda}$?
\end{question}

\section{Model-theoretic aspects} \label{s:model}

\emph{We emphasize that ``$\ulambda$-space'' means ``$\ulambda$-space over $k$'' unless otherwise mentioned.}

\subsection{Theories}

Fix a tuple $\ulambda=[\lambda_1, \ldots, \lambda_r]$. We now introduce the first-order language $\cL_{\ulambda}=\cL_{\ulambda,k}$ we use to describe $\ulambda$-spaces. This language is two-sorted: we use Greek symbols (such as $\alpha$, $\beta$) for scalar variables, and Roman symbols (such as $x$, $y$) for vector variables. The language contains function symbols for addition and multiplication of scalars, together with a constants for each element of $k$. It also contains function symbols for scalar-vector multiplication and vector addition, together with a constant symbol $\bzero$ for the zero vector. Finally, for each $1 \le i \le r$ there is a scalar-valued function symbol $\omega_i$ taking $\vert \lambda_i \vert$ vector inputs.

Suppose that $\phi$ is a formula with $m$ free scalars variables $\alpha_1, \ldots, \alpha_m$ and $n$ free vector variables $x_1, \ldots, x_n$. We then say that $\phi$ is a \defn{$(m,n)$-formula}. It will be convenient to package the variables into tuples $\ul{\alpha}=(\alpha_1, \ldots, \alpha_m)$ and $\ul{x}=(x_1, \ldots, x_n)$, and write $\phi(\ul{\alpha}, \ul{x})$ in place of $\phi(\alpha_1, \ldots, \alpha_m, x_1, \ldots, x_n)$.

Let $V$ be a $\ulambda$-space. Then $k \amalg V$ is a naturally structure for $\cL_{\ulambda}$. (It is important to remember that the scalars are part of the structure.) To define $\omega_i$, we convert the given $\lambda_i$-form on $V$ to a multilinear map as in \S \ref{ss:forms}. We let $\Th(V)$ be the theory of $V$; this is the set of all sentences in $\cL_{\ulambda}$ that are true for $V$. We say that two $\ulambda$-spaces $V$ and $W$ are \defn{elementarily equivalent} if $\Th(V)=\Th(W)$.

\begin{remark} \label{rmk:extn}
Suppose that $V'$ is a $\ulambda$-space over an extension field $k'$ of $k$. Then $k' \amalg V'$ is naturally a structure for $\cL_{\ulambda}$, and we let $\Th(V', k')$ be its theory. We will mostly not be concerned with this situation. However, one must keep it in mind, for if $V$ is a $\ulambda$-space then a model of the theory $\Th(V)$ is a $\ulambda$-space $V'$ over $k'$ such that $\Th(V', k')=\Th(V)$. Thus the model theory of $\Th(V)$ ``sees'' these examples.
\end{remark}

We now give two examples to illustrate some of the information first-order statements can detect about $\ulambda$-spaces.

\begin{example} \label{ex:lin-ind}
Consider the formula $\theta_2(x_1,x_2)$ given by
\begin{displaymath}
\forall \alpha_1, \alpha_2 ( \alpha_1 x_1+\alpha_2 x_2 =\bzero \implies \alpha_1=0 \land \alpha_2=0 )
\end{displaymath}
This formula expresses that $x_1$ and $x_2$ are linearly independent. Of course, there is a similar $(0,n)$-formula $\theta_n(\ul{x})$ expressing linear independence of $n$ vectors. We have $\dim(V) \ge n$ if and only if the sentence $\exists \ul{x} (\theta_n(\ul{x}))$ belongs to $\Th(V)$. Thus if $V$ and $W$ are elementarily equivalent then either $V$ and $W$ are both infinite dimensional, or $V$ and $W$ have equal finite dimension.
\end{example}

\begin{example} \label{ex:quad-thy}
Let $\ulambda=[(2)]$, so that we are working with quadratic spaces. Consider the formula $\psi(x)$ given by
\begin{displaymath}
\forall y(\omega(x,y)=0).
\end{displaymath}
This formula expresses that $x$ is in the null space of the form $\omega$. Thus the sentence
\begin{displaymath}
\exists x(x \ne \bzero \land \psi(x))
\end{displaymath}
means that the null space is non-zero. Similarly, letting $\theta_n$ be as in Example~\ref{ex:lin-ind}, the sentence
\begin{displaymath}
\exists x_1, \ldots, x_n (\theta_n(x_1, \ldots, x_n) \land \psi(x_1) \land \cdots \land \psi(x_n))
\end{displaymath}
means that there are $n$ linearly independent vectors in the null space. We thus see that if $V$ and $W$ are elementarily equivalent quadratic spaces then the null spaces of $V$ and $W$ are either both infinite dimensional, or have the same finite dimension.
\end{example}


\subsection{Types}

An important idea in model theory is the notion of \emph{type}. We will require a small amount of type theory, which we now discuss. We refer to \cite[\S 6.3]{Hodges} for additional background.

Let $\fT$ be a complete theory in the language $\cL_{\ulambda}$, e.g., the theory of some $\ulambda$-space. We say that two $(0,n)$-formulas $\phi(\ul{x})$ and $\psi(\ul{x})$ are \defn{equivalent} modulo $\fT$ if the sentence
\begin{displaymath}
\forall \ul{x} (\phi(\ul{x}) \iff \psi(\ul{x}))
\end{displaymath}
belongs to $\fT$. The set $R_n$ of equivalence classes forms a boolean algebra under conjunction and disjunction. Of course, one could make a more general definition that accommodates $(m,n)$-formulas, but we will not need this.

Let $V$ be a $\ulambda$-space with $\fT=\fT(V)$, and let $\ul{v} \in V^n$. The \defn{type} of $\ul{v}$, denoted $t(\ul{v})$, is the set of all $(0,n)$-formulas $\phi$ satisfied by $\ul{v}$. Note that if $\phi$ and $\psi$ are equivalent the $\phi(\ul{v})$ holds if and only if $\psi(\ul{v})$ holds. Thus $t(\ul{v})$ is a union of equivalence classes, and can therefore be regarded as a subset of $R_n$. The type of $\ul{v}$ determines the isomorphism type of the $\ulambda$-space $\operatorname{span}(\ul{v})$, but typically contains more information (related to how this space sits in $V$).

In fact, there is a more abstract notion of type: an \defn{$n$-type} of $\fT$ is a maximal ideal of the ring $R_n$. The type $t(\ul{v})$ of $\ul{v} \in V^n$ is a type in this sense: indeed, $t(\ul{v})$ is the kernel of the ring homomorphism $R_n \to \bF_2$ that takes $\phi$ to~0 if $\phi(\ul{v})$ holds, and~1 otherwise. We say that a type $t$ of $\fT$ \defn{occurs} in $V$ if $t=t(\ul{v})$ for some $\ul{v} \in V^n$.

We now look at a few special classes of types.

\subsubsection{Linearly independent types}

We say that an $n$-type is \defn{linearly independent} if it contains the formula $\theta_n$ from Example~\ref{ex:lin-ind}. Of course, if $\ul{v} \in V^n$ then $t(\ul{v})$ is linearly independent if and only if $\ul{v}$ is. It typically suffices to study linearly independent types.

%

\subsubsection{Principal types}

We say that an $n$-type $t$ is \defn{principal} if it is axiomatized by a single formula $\kappa(\ul{x}) \in t$. This means that $t$ consists of exactly those formulas $\phi(\ul{x})$ for which the sentence $\forall \ul{x} (\kappa(\ul{x}) \implies \phi(\ul{x}))$ belongs to $\fT$. Equivalently, it means that $t$ is the principal ideal of $R_n$ generated by $\kappa(\ul{x})$. Principal types are especially easy to work with: if $\kappa(\ul{x})$ axiomatizes $t$ then $\ul{v} \in V^n$ has type $t$ if and only if $\kappa(\ul{v})$ holds.

\begin{example}
Suppose that $\ulambda=\emptyset$, so we just have vector spaces, and let $\fT$ be the theory of a countable vector space. The following four formulas axiomatize principal types:
\begin{displaymath}
x=\bzero, \qquad x \ne \bzero, \qquad \theta_2(x,y), \qquad (x \ne \bzero) \land (y=2x).
\end{displaymath}
Here $\theta_2$ tests linear independence (see Example~\ref{ex:lin-ind}). It is not difficult to prove that these formulas axiomatize principal types (one can argue as in the proof of Proposition~\ref{prop:homo-type}.)
\end{example}

The following proposition demonstrates the usefulness of principal types:

\begin{proposition} \label{prop:atomic}
Let $V$ and $W$ be countable $\ulambda$-spaces that are elementarily equivalent and in which all linearly independent types are principal. Then $V$ and $W$ are isomorphic.
\end{proposition}

\begin{proof}
Suppose that we have linearly independent tuples $\ul{v} \in V^n$ and $\ul{w} \in W^n$ such that $t(\ul{v})=t(\ul{w})$, and let $v_{n+1}$ be another element of $V$ that is linearly independent of $\ul{v}$. We show that there exists $w_{n+1} \in W$ such that $t(\ul{v}, v_{n+1})=t(\ul{w}, w_{n+1})$. This implies that $w_{n+1}$ is linearly independent of $\ul{w}$. Of course, the analogous statement with $V$ and $W$ switched will then be true by symmetry.

Let $\psi(\ul{x}, x_{n+1})$ axiomatize $t(\ul{v}, v_{n+1})$. Since $\exists x_{n+1} (\psi(\ul{v}, x_{n+1}))$ is true ($v_{n+1}$ is a witness), it follows that the formula $\exists x_{n+1} (\psi(\ul{x}, x_{n+1}))$ belongs to $t(\ul{v})$. It therefore belongs to $t(\ul{w})$, and so $\exists x_{n+1}(\psi(\ul{w}, x_{n+1}))$ is true. Let $w_{n+1}$ be a witness. Since $\psi(\ul{w}, w_{n+1})$ holds and $\psi$ axiomatizes $t(\ul{v}, v_{n+1})$, it follows that $t(\ul{w}, w_{n+1})=t(\ul{v}, v_{n+1})$.

The result now follows from a back and forth argument, as follows. Fix bases $X$ and $Y$ of $V$ and $W$ indexed by $\bN$. We will construct new bases $(v_i)_{i \ge 1}$ and $(w_i)_{i \ge 1}$ of $V$ and $W$ such that $t(v_1, \ldots, v_n)=t(w_1, \ldots, w_n)$ for all $n$. This implies that $v_i \mapsto w_i$ is an isomorphism of $\ulambda$-spaces. We construct our new bases inductively. Suppose we have constructed $(v_1, \ldots, v_n)$ and $(w_1, \ldots, w_n)$. We then do the following two steps:
\begin{itemize}
\item Let $v_{n+1}$ be the first basis vector in $X$ not in the span of $(v_1, \ldots, v_n)$, and let $w_{n+1}$ be a vector in $W$ such that $t(w_1, \ldots, w_{n+1})=t(v_1, \ldots, v_{n+1})$.
\item Let $w_{n+2}$ be the first basis vector in $Y$ not in the span of $(w_1, \ldots, w_{n+1})$, and let $v_{n+2}$ be a vector in $V$ such that $t(v_1, \ldots, v_{n+2})=t(w_1, \ldots, w_{n+2})$.
\end{itemize}
The choice of $v_{n+1}$ ensures that $(v_1, \ldots, v_{n+1})$ is linearly independent, and so $(w_1, \ldots, w_{n+1})$ is linearly independent too. Similarly for the second step. Since every element of $X$ belongs to the span of $(v_1, \ldots, v_n)$ for some $n$, it follows that $(v_i)_{i \ge 1}$ is indeed a basis; similarly for $(w_i)_{i \ge 1}$.
\end{proof}

\begin{remark}
The above proof is an adaptation of a standard argument \cite[Theorem~7.2.3]{Hodges} to the linear case.
\end{remark}

\begin{remark}
Let $V$ be a $\ulambda$-space in which all linearly independent types are principal. One can then show that all types in $V$ are principal. In fact, this even holds for $(m,n)$-types, i.e., types involving scalars. In the terminology of model theory, $V$ is an \emph{atomic structure}. See \cite[\S 7.2]{Hodges} for more.
\end{remark}

\subsubsection{Rational types}

Every type appears in some model of $\fT$. However, we only care about models where the scalar field is $k$. With this in mind, we say that an $n$-type $t$ is \defn{rational} if there exists a $\ulambda$-space $V$ with $\Th(V)=\fT$ such that $t$ occurs in $V$. (We re-emphasize that $V$ is over $k$.) We note that even is $k$ is algebraically closed, irrational types will typically exist.

\begin{example}
Here is an example of an irrational 2-type. Take $\ulambda$ to be empty and $\fT$ to be the theory of a countable dimensional $k$-vector space. Start with the formulas
\begin{displaymath}
\exists \alpha (x = \alpha y), \qquad x \ne \bzero, \qquad y \ne \bzero.
\end{displaymath}
Thus we are looking at a pair of non-zero linearly dependent vectors. Of course, the scalar $\alpha$ in the first equation is unique. Now let $t_0$ be a 1-type in the theory of $k$ (as a field, or really, field extension of $k$), which is not the type of $0 \in k$. For each formula $\psi(\alpha)$ in $t_0$, add the formula
\begin{displaymath}
\exists \alpha (x=\alpha y \land \psi(\alpha))
\end{displaymath}
to our list of formulas. Now take the 2-type $t$ axiomatized by all the above formulas. This expresses that $x$ and $y$ are linearly dependent, and that the scalar relating them has type $t_0$. If $t_0$ does not occur as a type in $k$ then $t$ will not be a rational type. (To produce such a $t_0$, let $k^*$ be an ultrapower of $k$, let $a$ be an element of $k^*$ that does not belong to $k$, and take $t_0$ be the type of $a$.)
\end{example}

\subsection{The classifying map} \label{ss:class}

We now pause our discussion of model theory to introduce a useful tool. Let $(V, \uomega)$ be a $\ulambda$-space. Fix a a non-negative integer $n$, and let $V^{[n]}$ be the subset of $V^n$ consisting of tuples $(v_1, \ldots, v_n)$ that are linearly independent. Let $X=\bS_{\ulambda}(k^n)^*$ be the space of $\ulambda$-structures on $k^n$. Given $\ul{v} \in V^{[n]}$, we obtain an injective linear map $j_{\ul{v}} \colon k^n \to V$ by mapping the $i$th basis vector $e_i$ of $k^n$ to $v_i$. We define a function
\begin{displaymath}
\pi \colon V^{[n]} \to X, \qquad \pi(\ul{v}) = j_{\ul{v}}^*(\uomega).
\end{displaymath}
We call $\pi$ the \defn{classifying map}.

Note that $X$ is a finite dimensional $k$-vector space. If we fix a basis, then the components of $\pi(\ul{v})$ are obtained by evaluating the given $\lambda_i$-forms on the components of $\ul{v}$ (and taking linear combinations). We thus see that $\pi$ is expressible in the language $\cL_{\ulambda}$.

\begin{example}
Suppose $\ulambda=[(2)]$. We can then identify $X$ with the space of symmetric $n \times n$ matrices. Let $x_{i,j}$ with $1 \le i \le j \le n$ be the coordinates on this space. Then $\pi_{i,j}(\ul{v})=\omega(v_i, v_j)$, where $\pi_{i,j}$ denotes the $(i,j)$ coordinate of $\pi$.
\end{example}

We need a few simple properties of $\pi$:

\begin{proposition} \label{prop:class}
Maintain the above notation, and let $G=\Aut(V)$.
\begin{enumerate}
\item The map $\pi$ is $G$-invariant, i.e., $\pi(\ul{v})=\pi(g \ul{v})$ for $g \in G$.
\item If $V$ is universal then $\pi$ is surjective.
\item If $V$ is homogeneous then $\pi$ is injective modulo $G$, i.e., $\pi(\ul{v})=\pi(\ul{w})$ if and only if $\ul{v}=g \ul{w}$ for some $g \in G$.
\end{enumerate}
\end{proposition}

\begin{proof}
(a) We have
\begin{displaymath}
\pi(g\ul{v}) = j_{g\ul{v}}^*(\uomega) = (gj_{\ul{v}})^*(\uomega) = j_{\ul{v}}^* g^* (\uomega) = j_{\ul{v}}^*(\uomega) = \pi(\ul{v}).
\end{displaymath}
Here $g^*(\uomega)=\uomega$ since $g$ is an automorphism of $V$.

(b) Let $\ueta \in X$ be a $\ulambda$-structure on $k^n$. Since $V$ is universal, there is an embedding of $\ulambda$-spaces $j \colon (k^n, \ueta) \to (V, \uomega)$; this means $j^*(\uomega)=\ueta$. Letting $v_i=j(e_i)$, we have $j=j_{\ul{v}}$, and so $\pi(\ul{v})=\ueta$.

(c) Suppose $\pi(\ul{v})=\pi(\ul{w})$. This exactly means that the linear isomorphism $\operatorname{span}(\ul{v}) \to \operatorname{span}(\ul{w})$ taking $v_i$ to $w_i$ is an isomorphism of $\ulambda$-spaces. Since $V$ is homogeneous, this isomorphism extends to an automorphism $g$ of $G$. Clearly, $\ul{v}=g \ul{w}$, which proves the claim.
\end{proof}

\subsection{Types in homogeneous spaces} \label{ss:homo-type}

We now examine $n$-types in the universal homogeneous countable $\ulambda$-space $V=V_{\ulambda}$. Let $\pi \colon V^{[n]} \to X$ be the classifying map (\S \ref{ss:class}), with $X=\bS_{\ulambda}(k^n)^*$. Suppose that $X$ is $m$-dimensional (as a vector space), and fix a basis. Let $\kappa(\ul{\alpha}, \ul{x})$ be the $(m,n)$-formula such that
\begin{displaymath}
\kappa(\ul{a}, \ul{v}) \iff \theta_n(\ul{v}) \land \pi(\ul{v})=\ul{a},
\end{displaymath}
where $\theta_n$ is as in Example~\ref{ex:lin-ind}. In other words, $\kappa(\ul{a}, \ul{v})$ holds if and only if $\ul{v}$ is linearly independent and $\pi(\ul{v})$ is the point $\ul{a}$ of $X$. We let $\kappa_{\ul{a}}(\ul{x})$ denote the $(0,n)$-formula $\kappa(\ul{a}, \ul{x})$.

\begin{proposition} \label{prop:homo-type}
We have the following:
\begin{enumerate}
\item For $\ul{a} \in X$, the formula $\kappa_a(\ul{x})$ axiomatizes a principal type $t_{\ul{a}}$ of $\Th(V)$ which occurs in $V$.
\item The linearly independent rational types of $\Th(V)$ are exactly the $t_{\ul{a}}$ with $\ul{a} \in X$.
\end{enumerate}
\end{proposition}

\begin{proof}
(a) Let $\ul{a} \in X$ be given. Choose $\ul{v} \in V^{[n]}$ with $\pi(\ul{v})=\ul{a}$, which is possible since $\pi$ is surjective (Proposition~\ref{prop:class}). It is clear that $\kappa_{\ul{a}}(\ul{x})$ belongs to $t(\ul{v})$. We claim that it axiomatizes it. Suppose that $\phi(\ul{x})$ belongs to $t(\ul{v})$. We must show that the sentence
\begin{displaymath}
\forall \ul{x} (\kappa_{\ul{a}}(\ul{x}) \implies \phi(\ul{x}))
\end{displaymath}
belongs to $\Th(V)$. To verify this, it is enough to show that $\kappa_a(\ul{w})$ implies $\phi(\ul{w})$ for all $\ul{w} \in V^n$. Thus suppose $\kappa_a(\ul{w})$ holds. Then $\ul{w}$ is linearly independent and $\pi(\ul{w})=\ul{a}=\pi(\ul{v})$. Since $\pi$ is injective modulo $G_{\ulambda}$ (Proposition~\ref{prop:class}), we have $\ul{w}=g \ul{v}$ for some $g \in G_{\ulambda}$. Since $\phi$ is invariant under $G_{\ulambda}$ and $\phi(\ul{v})$ holds, it follows that $\phi(\ul{w})$ holds. 

(b) Let $t$ be a linearly independent rational type of $\Th(V)$. By definition, there is some $\ulambda$-space $W$ with $\Th(W)=\Th(V)$ and some $\ul{w} \in W^n$ such that $t=t(\ul{w})$. Let $\pi' \colon W^{[n]} \to X$ be the classifying map, and let $\ul{a}=\pi'(\ul{w})$. Then $\kappa_{\ul{a}}(\ul{w})$ holds, and so $t(\ul{w})=t_{\ul{a}}$.
\end{proof}

\subsection{Categoricity}

In classical model theory, a countable structure $X$ is called \defn{$\omega$-categorical} if any other countable structure that is elementarily equivalent to $X$ is actually isomorphic to $X$. This has proved to be an important concept, and is closely related to homogeneity (most countable homogeneous structures of interest are $\omega$-categorical). It therefore makes sense to examine the idea in the linear setting.

The usual concept of $\omega$-categorical is not the right thing to consider for $\ulambda$-spaces. There are two problems. The first is that if $k$ is uncountable then there are no relevant countable structures. (The field $k$ is technically part of the structure, so even the zero space leads to an uncountable structure.) The second problem relates to the issue raised in Remark~\ref{rmk:extn} concerning the coefficient field. We illustrate this with an example. Suppose $k=\bQ$  and $\ulambda$ is empty so that we just have vector spaces. We would like $k^{\oplus \infty}$ (countable sum) to count as $\omega$-categorical. However, if $k' \ne \bQ$ is a countable field that is elementarily equivalent to $k$ as field then the $\cL_{\ulambda}$-structures associated to $(k')^{\oplus \infty}$ and $k^{\oplus \infty}$ are elementarily equivalent but not isomorphic. Since such $k'$ do exist, it follows that $k^{\oplus \infty}$ is not $\omega$-categorical.

Due to the above issues, we introduce the following variant of the $\omega$-categorical concept in our setting:

\begin{definition}
Let $V$ be a countable $\ulambda$-space. We say that $V$ is \defn{linearly $\omega$-categorical} if any countable $\ulambda$-space that is elementarily equivalent to $V$ is isomorphic to $V$. 
\end{definition}

We emphasize that in the above definition, ``countable'' means ``dimension $\aleph_0$,'' and both $V$ and $W$ are over $k$. The following is our main theorem related to this concept:

\begin{theorem} \label{thm:categorical}
The space $V_{\ulambda}$ is linearly $\omega$-categorical.
\end{theorem}

\begin{proof}
Let $W$ be a countable $\ulambda$-space that is elementarily equivalent to $V$. Since every linearly independent rational type of $\Th(V)=\Th(W)$ is principal (Proposition~\ref{prop:homo-type}), it follows that all linearly independent types in $V$ and $W$ are principal. Thus $V$ and $W$ are isomorphic (Proposition~\ref{prop:atomic}), and so $V$ is linearly $\omega$-categorical.
\end{proof}

\begin{remark}
There are other examples of linearly $\omega$-categorical spaces. For example, if $k$ is algebraically closed then any countable quadratic space is linearly $\omega$-categorical. One can see this using the observation in Example~\ref{ex:quad-thy} and some related ideas. This example is also discussed (from a different perspective) in \cite{Kamsma}.
\end{remark}

\begin{remark} \label{rmk:ryll}
Let $\Omega$ be a countable structure over a countable language. The classical Ryll-Nardzewski theorem asserts that the following conditions are equivalent:
\begin{itemize}
\item $\Omega$ is $\omega$-categorical.
\item $\Th(\Omega)$ has finitely many $n$-types for all $n$.
\item $\Aut(\Omega)$ is oligomorphic.
\end{itemize}
See \cite[\S 7.3]{Hodges} for a more complete statement.

It is natural to look for a linear analog of the Ryll-Nardzewski theorem. Let $V$ be a countable $\ulambda$-space. Consider the following conditions:
\begin{itemize}
\item $V$ is linearly $\omega$-categorical.
\item The $n$-types in $\Th(V)$ form a finite dimensional space, for all $n$.
\item $\Aut(V)$ is linear-oligomorphic.
\end{itemize}
Are these conditions equivalent? The precise meaning of the second condition is not clear to us yet. For $V_{\ulambda}$, we showed that the linearly independent rational $n$-types are the $k$-points of the finite dimensional variety $X=\bS_{\ulambda}(k^n)$ (Proposition~\ref{prop:homo-type}), which confirms this condition to some extent.
\end{remark}

\begin{remark}
One can also define linearly $\kappa$-categorical, for other infinite cardinals $\kappa$: the $\ulambda$-spaces involved should have dimension $\kappa$. We do not know what happens in uncountable cardinalities. The behavior of quadratic spaces in dimension $\aleph_1$ (see Example~\ref{ex:quad-space}) suggests the situation could be very different.
\end{remark}

\subsection{Quantifier elimination} \label{ss:quant}

Recall that a first-order theory $\fT$ has \defn{quantifier elimination} if every formula is equivalent (modulo $\fT$) to a quantifier-free formula, and a structure has \defn{quantifier elimination} if its theory does. In general, $\ulambda$-spaces do not have quantifier elimination. There are two issues. First, $k$ itself may not have quantifier elimination (as a field), and this prevents $\Th(V)$ from having quantifier elimination. Second, the formula $\theta_n$ from Example~\ref{ex:lin-ind} expressing linear independence will typically not be equivalent to a quantifier-free formula, for any $k$.

In both issues above, the problematic quantifiers are over scalar variables. This suggests the following definition:

\begin{definition}
We say that a theory $\fT$ for $\cL_{\ulambda}$ has \defn{vector-quantifier elimination} if every formula is equivalent (modulo $\fT$) to one involving no quantifiers over vector variables. We say that a $\ulambda$-space $V$ has \defn{vector-quantifier elimination} if $\Th(V)$ does.
\end{definition}

The following is our main result in this direction:

\begin{theorem} \label{thm:quant}
The space $V_{\ulambda}$ has vector-quantifier elimination.
\end{theorem}

We require several lemmas before giving the proof. We say that a subset $S$ of $k^n$ is a \defn{$D$-(sub)set} if it is definable in the language of fields (using constants from $k$). This is typically called a ``definable subset,'' but this terminology could be ambiguous in our setting (since we have both the language of fields and of $\ulambda$-spaces). Note that the notion of D-set is invariant under $\GL_n(k)$, and so it therefore makes sense for subsets of finite dimensional vector spaces.

\begin{lemma} \label{lem:quant-1}
Let $V$ be a finite $\ulambda$-space, let $\phi(\ul{\alpha}, \ul{x})$ be an $(m,n)$-formula, and let $K \subset k^m \times V^n$ be the set satisfying $\phi$. Then $K$ is a D-set.
\end{lemma}

\begin{proof}
Choosing a basis, identify $V$ with $k^d$, and regard $K$ as a subset of $k^m \times k^{nd}$. The $\lambda_i$-forms on $V$ amount to polynomial maps $k^d \to k$, which makes the lemma clear.
\end{proof}

\begin{lemma} \label{lem:quant-2}
Let $V$ be a $\ulambda$-space, let $G=\Aut(V)$, and assume $G(\ul{u})$ is linear-oligomorphic for all $\ul{u} \in V^m$. Let $\phi(\ul{\alpha}, \ul{x})$ be an $(m,n)$-formula, and let $L \subset k^m \times V^n$ be the set satisfying $\phi(\ul{\alpha}, \ul{x})$. Let $W$ be a finite dimensional subspace of $V$, and let $K=L \cap (k^m \times W^n)$. Then $K$ is a D-subset of $k^m \times W^n$.
\end{lemma}

\begin{proof}
We prove the proposition assuming that $\phi$ has the form $\exists y (\psi(\ul{\alpha}, \ul{x}, y))$ where $\psi$ has no vector quantifiers. This is sufficient for the application of the lemma, and the proof in the general case is not much different.

Let $H \subset \Aut(V)$ be the subgroup fixing each element of $W$. By assumption, $H$ is a linear-oligomorphic subgroup of $\GL(V)$. Let $W \subset V'$ be a finite dimensional subspace of $V$ such that every $H$-orbit on $V$ meets $V'$. Regard $V'$ as a finite $\ulambda$-space, and let $L' \subset k^m \times (V')^n$ be the set satisfying $\phi$; here the existential quantifier in $\phi$ only ranges over $V'$. The set $L'$ is a D-set by Lemma~\ref{lem:quant-1}.

We claim that $K=L' \cap (k^m \times W^n)$, which will prove that $K$ is a D-set. Since $L' \subset L$, it is clear that the right side is contained in $K$. Suppose now that $(\ul{a}, \ul{v})$ is an element of $K$. Then there is some $w \in V$ such that $\psi(\ul{a}, \ul{v}, w)$ holds. Let $h \in H$ be such that $hw \in V'$. Since $h$ fixes $\ul{v}$ and the veracity of $\psi$ is unaffected by applying $h$, we see that $\psi(\ul{a}, \ul{v}, hw)$ holds. In other words, we can find a witness to $\phi(\ul{a}, \ul{v})$ in $V'$, and so $(\ul{a}, \ul{v})$ belongs to $L'$. This completes the proof.
\end{proof}

Let $\phi(\ul{\alpha}, \ul{x})$ be an $(m,n)$-formula. We say that $\phi$ satisfies (LI) if $\phi(\ul{a}, \ul{v})$ implies that the tuple $(v_1, \ldots, v_n)$ is linearly independent.

\begin{lemma} \label{lem:quant-3}
Let $V$ be a $\ulambda$-space. Suppose every $(m,n)$-formula satisfying (LI) is equivalent modulo $\Th(V)$ to one without vector-quantifiers. Then $V$ has vector-quantifier elimination.
\end{lemma}

\begin{proof}
Let $\phi$ be a $(m,n)$-formula. We show that $\phi$ is equivalent to a formula without vector quantifiers. 
We just treat the case $(m,n)=(0,2)$ for notational simplicity; the idea in general is the same. Let $\theta_n$ be the $(0,n)$-formula from Example~\ref{ex:lin-ind} that detects linear independence. We introduce the following three formulas:
\begin{align*}
\phi_1(x,y) &\colon \theta_2(x,y) \land \phi(x,y) \\
\phi_2(\alpha,x) &\colon \theta_1(x) \land \phi(x, \alpha x) \\
\phi_2(\beta,y) & \colon \theta_1(y) \land \phi(\beta y, y)
\end{align*}
We have
\begin{align*}
\phi(x,y) \iff & \phi_1(x,y) \\
& \lor \exists \alpha(y = \alpha x \land \phi_2(\alpha, x)) \\
& \lor \exists \beta(x=\beta y \land \phi_3(\beta, y)) \\
& \lor (x=y=\bzero \land \phi(\bzero, \bzero)).
\end{align*}
This essentially breaks up $\phi(x,y)$ into cases depending on the possible linear dependencies between $x$ and $y$. The formulas $\phi_1$, $\phi_2$, and $\phi_3$ satisfy (LI), and are therefore equivalent to formulas without vector-quantifiers. By the above, we see that $\phi$ is also equivalent to a formula without vector-quantifiers.
\end{proof}

\begin{proof}[Proof of Theorem~\ref{thm:quant}]
Let $\phi(\ul{\alpha}, \ul{x})$ be an $(m,n)$-formula satisfying (LI). We show that $\phi$ is equivalent to a formula having no vector-quantifiers. This will establish the theorem by Lemma~\ref{lem:quant-3}. We note that it suffices (by induction on the number of vector quantifiers) to treat the case where $\phi$ has the form $\exists y(\psi(\ul{\alpha}, \ul{x}, y))$ where $\psi$ has no vector quantifiers, though we do not use this.

Let $V=V_{\ulambda}$ and $G=G_{\ulambda}$. Let $\pi_0 \colon V^{[n]} \to X$ be the classifying map (see \S \ref{ss:class}), with $X=\bS_{\ulambda}(k^n)^*$. Let $\pi \colon k^m \times V^{[n]} \to k^m \times X$ be the map $\id \times \pi_0$. We use Proposition~\ref{prop:class} without mention in what follows. Let $W$ be a finite dimensional subspace of $V$ such that every $G$-orbit on $V^{[n]}$ meets $W^{[n]}$. This exists since $G$ is linear-oligomorphic (Theorem~\ref{thm:olig}).

Let $L \subset k^m \times V^{[n]}$ be the set satisfying $\phi$, and let $K=L \cap (k^m \times W^{[n]})$. The set $K$ is a D-set by Lemma~\ref{lem:quant-2}; note that $G$ satisfies the assumption of that lemma by Proposition~\ref{prop:olig2}. Thus $K'=\pi(K)$ is a D-subset of $k^m \times X$. Let $\phi'(\ul{\alpha}, \ul{x})$ be a formula without vector quantifiers expressing that $\ul{x}$ is linear independent and $\pi(\ul{\alpha}, \ul{x}) \in K'$.

Let $(\ul{a}, \ul{v}) \in k^m \times V^{[n]}$. We claim that $\phi(\ul{a}, \ul{v})$ holds if and only if $\phi'(\ul{a}, \ul{v})$ holds; this will prove the theorem. First suppose that $\phi(\ul{a}, \ul{v})$ holds. Let $g \in G$ be an element such that $g \ul{v} \in W^{[n]}$.  Since $\phi$ is $G$-invariant, it follows that $\phi(\ul{a}, g\ul{v})$ holds, and so $(\ul{a}, g\ul{v}) \in K$. Hence $\pi(\ul{a}, g\ul{v}) \in K'$. But $\pi$ is $G$-invariant, and so $\pi(\ul{a}, \ul{v}) \in K'$, that is, $\phi'(\ul{a}, \ul{v})$ holds.

Now suppose that $\phi'(\ul{a}, \ul{v})$ holds, meaning $\pi(\ul{a}, \ul{v}) \in K'$. It follows that there is some $(\ul{b}, \ul{w}) \in K$ such that $\pi(\ul{a}, \ul{v})=\pi(\ul{b}, \ul{w})$, i.e., $\ul{a}=\ul{b}$ and $\pi_0(\ul{v})=\pi_0(\ul{w})$. Since $\pi_0$ is injective modulo $G$, there is some $g \in G$ such that $\ul{v}=g \ul{w}$. Since $\phi(\ul{b}, \ul{w})$ holds and $\phi$ is $G$-invariant, it follows that $\phi(\ul{a}, \ul{v})$ holds. This completes the proof.
\end{proof}

\begin{remark}
Suppose $k$ has quantifier elimination, e.g., $k$ is algebraically closed or real closed (Tarski). Then the D-set $K'$ in the above proof can be described in the language of fields without quantifiers, and so $\phi'$ can be chosen without quantifiers. This shows that we only need quantifiers to test linear independence.

We can reformulate this as follows. Let $\cL'_{\ulambda}$ be the language obtained from $\cL_{\ulambda}$ by adding relation symbols that test for linear independence and function symbols that give the coefficients when expressing one vector as a linear combination of a tuple of linear independent vectors. Then the theory of $V_{\ulambda}$ over the language $\cL'_{\ulambda}$ has quantifier elimination for vectors and scalars.
\end{remark}

\begin{remark} \label{rmk:decidable}
Assume $k=\ol{\bQ}$. One can then show that $\Th(V_{\ulambda})$ is decidable. The key point is that everything in the proof of Theorem~\ref{thm:quant} is effective, and so we have effective elimination of quantifiers. Perhaps the most subtle point is the construction of $W$, and what it even means to do computations in $V_{\ulambda}$. For this, we must look at the construction of the Fra\"iss\'e limit in \S \ref{ss:fraisse}. In the case of $\ulambda$-spaces, these constructions are effective; this hinges on the explicit construction of universal $\ulambda$-spaces. This construction therefore produces a model of $V_{\ulambda}$ in which we can do computations. To find $W$, one finds an embedding of a finite $n$-universal space into $V_{\ulambda}$.
\end{remark}

\begin{remark} \label{rmk:qe-converse}
There is a converse to Theorem~\ref{thm:quant}: if $V$ is a $\ulambda$-space with vector-quantifier elimination then $V$ is homogeneous. We sketch the proof. Let $V$ be given. Suppose $\phi(\ul{x})$ is an $(0,n)$-formula with no vector quantifiers. If $\ul{v} \in V^{[n]}$ then the veracity of $\phi(\ul{v})$ depends only on $\pi(\ul{v}) \in X$. It follows that $\pi(\ul{v})$ completely determines the type $t(\ul{v})$, and so (as in \S \ref{ss:homo-type}) every linearly independent type in $V$ is principal. Suppose $\ul{v}, \ul{w} \in V^{[n]}$ satisfy $\pi(\ul{v})=\pi(\ul{w})$, i.e., we have an isomorphism $\operatorname{span}(\ul{v}) \to \operatorname{span}(\ul{w})$ of pinned $\ulambda$-spaces. The proof of Proposition~\ref{prop:atomic} shows that we can find an automorphism $g$ of $V$ such that $g \ul{v}=\ul{w}$. This shows that $V$ is homogeneous.
\end{remark}

\appendix
\section{Fra\"iss\'e theory} \label{s:fraisse}

\subsection{Overview}

Fra\"iss\'e \cite{Fraisse} proved an important theorem that, roughly speaking, explains when a collection of finite relational structures can be assembled to a countable homogeneous structure. See \cite[\S 2.6]{Cameron} for an expository treatment. We want to apply Fra\"iss\'e's theorem to construct homogeneous $\ulambda$-spaces. However, since $\ulambda$-spaces are not finite, and may not even be countable, the classical form of the theorem does not apply. In this appendix, we formulate a generalization of Fra\"iss\'e's theorem that applies in our setting. We use the language of category theory since it seems to be the most flexible and convenient.

We do not claim any originality here: categorical formulations of Fra\"iss\'e's theorem have been known since the work of Droste--G\"obel \cite{DrosteGobel1,DrosteGobel2}, and have appeared in more recent work as well \cite{Caramello,Irwin,Kubis}. We have included this material simply for the convenience of the reader.

\subsection{Ind-objects} \label{ss:ind}

Fix, throughout \S \ref{s:fraisse}, a category $\cC$ in which all morphisms are monic. We often refer to morphisms in $\cC$ as embeddings. The main case to keep in mind is where $\cC$ is a category of finite $\ulambda$-spaces.

An \emph{ind-object} of $\cC$ is a diagram
\begin{displaymath}
X_1 \to X_2 \to X_3 \to \cdots
\end{displaymath}
More formally, an ind-object is a pair $X=(X_{\bullet}, \epsilon_{\bullet,\bullet})$ where $X_i$ is an object of $\cC$ for $i \ge 1$, and $\epsilon_{i,j} \colon X_i \to X_j$ is a morphism for $i \le j$, such that $\epsilon_{i,i}$ is the identity and $\epsilon_{j,k} \circ \epsilon_{i,j} = \epsilon_{i,k}$. Of course, one can consider ind-objects indexed by other filtered categories, but we do not here.

Let $Y=(Y_{\bullet}, \delta_{\bullet,\bullet})$ be a second ind-object. A \emph{morphism} of ind-objects $\alpha \colon X \to Y$ is specified by a non-decreasing function $a \colon \bZ_+ \to \bZ_+$ and morphisms $\alpha_i \colon X_i \to Y_{a(i)}$ for each $i \in \bZ_+$ such that $\alpha_j \circ \epsilon_{i,j} = \delta_{a(i),a(j)} \circ \alpha_i$ for all $i \le j$. Suppose $b \colon \bZ_+ \to \bZ_+$ is a second non-decreasing function such that $b(i) \ge a(i)$ for all $i$. Define $\beta_i \colon X_i \to Y_{b(i)}$ by $\beta_i=\delta_{a(i),b(i)} \circ \alpha_i$. Then $(a,\alpha_{\bullet})$ and $(b,\beta_{\bullet})$ represent the same morphism $X \to Y$.

An object $X$ of $\cC$ is identified with the constant ind-object $X \to X \to \cdots$, where all transition maps are the identity. In this way we can talk about embeddings $X \to Y$ where $Y$ is an ind-object. Any such embedding factors through $Y_n$ for some $n$.

\begin{example}
Suppose $\cC$ is the category of finite dimensional vector spaces, with morphisms being injective linear maps. If $X$ is an ind-object in $\cC$ then we can associate to it the vector space $\Phi(X) = \varinjlim X_n$, which has dimension $\le \aleph_0$. If $Y$ is a second ind-object then giving an embedding $X \to Y$ of ind-objects is equivalent to giving an injective linear map $\Phi(X) \to \Phi(Y)$. In this way, $\Phi$ provides an equivalence between the category of ind-objects in $\cC$ and the category of vector spaces of dimension $\le \aleph_0$ (with injective maps).
\end{example}

\begin{example} \label{ex:ind-lambda}
Let $\cC=\cC^{\rf}_{\ulambda}$ be the category of finite $\ulambda$-spaces. Then, just as above, ind-objects in $\cC^{\rf}_{\ulambda}$ are equivalent to $\ulambda$-spaces of dimension $\le \aleph_0$.
\end{example}

\subsection{Universal objects}

We now introduce an important class of ind-objects:

\begin{definition}
An ind-object $\Omega$ of $\cC$ is \emph{universal} if every object of $\cC$ embeds into $\Omega$.
\end{definition}

We now establish a result that characterizes when a universal ind-object exists. To this end, we introduce the following conditions on $\cC$:
\begin{itemize}[leftmargin=1.5cm]
\item[(CC)] \defn{Countable cofinality}: there is a cofinal sequence of objects, i.e., there are objects $\{X_n\}_{n \ge 1}$ such that for any object $Y$ there is an embedding $Y \to X_n$ for some $n$.
\item[(JEP)] \defn{Joint embedding property}: given objects $X$ and $Y$, there is an object $Z$ and embeddings $X \to Z$ and $Y \to Z$.
\end{itemize}
The result is the following:

\begin{proposition} \label{A:univ}
The category $\cC$ has a universal ind-object if and only if it satisfies (CC) and (JEP).
\end{proposition}

\begin{proof}
Suppose $\cC$ satisfies (CC) and (JEP). Let $\{X_n\}_{n \ge 1}$ be a cofinal sequence of objects as in (CC). Now choose a diagram
\begin{displaymath}
\xymatrix{
\Omega_1 \ar[r] & \Omega_2 \ar[r] & \Omega_3 \ar[r] & \cdots \\
X_1 \ar[u] & X_2 \ar[u] & X_3 \ar[u] }
\end{displaymath}
We can construct such a diagram as follows. Take $\Omega_1=X_1$. Having defined $\Omega_{n-1}$, define $\Omega_n$ to be an object into which both $\Omega_{n-1}$ and $X_n$ embed; this exists by (JEP). Then $\Omega$ is an ind-object, and clearly universal: indeed, any object embeds into some $X_n$, which in turn embeds into $\Omega$.

Now suppose that $\cC$ has a universal ind-object $\Omega$. Given any object $X$, there is an embedding $X \to \Omega$, which factors through some $\Omega_n$. Thus any object embeds into some $\Omega_n$, and so (CC) holds. In fact, since $\Omega_n$ embeds into $\Omega_m$ for any $m>n$, we see that any $X$ embeds into $\Omega_n$ for any $n \gg 0$, from which (JEP) follows.
\end{proof}

\subsection{Homogeneous and f-injective objects}

We now introduce two additional important classes of ind-objects:

\begin{definition}
An ind-object $\Omega$ is \defn{homogeneous} if the following condition holds: given objects $X$ and $Y$ of $\cC$, embeddings $\alpha \colon X \to \Omega$ and $\beta \colon Y \to \Omega$, and an isomorphism $\gamma \colon X \to Y$, there exists an automorphism $\sigma \colon \Omega \to \Omega$ such that the diagram
\begin{displaymath}
\xymatrix{
\Omega \ar[r]^{\sigma} & \Omega \\
X \ar[u]^{\alpha} \ar[r]^{\gamma} & Y \ar[u]_{\beta} }
\end{displaymath}
commutes.
\end{definition}

\begin{definition}
An ind-object $\Omega$ is \defn{f-injective} if the following condition holds: given an embedding $\alpha \colon X \to Y$ in $\cC$ and an embedding $\gamma \colon X \to \Omega$, there exists an embedding $\beta \colon Y \to \Omega$ making the diagram
\begin{displaymath}
\xymatrix@C=4em{
& \Omega \\
X \ar[r]^{\alpha} \ar[ru]^{\gamma} & Y \ar[u]_{\beta} }
\end{displaymath}
commute.
\end{definition}

We note that if $\cC$ has an initial object then any f-injective ind-object is automatically universal: in the above diagram, take $X$ to be an initial object and $Y$ to be an arbitrary object. The following proposition connects the above concepts:

\begin{proposition} \label{prop:finj}
We have the following:
\begin{enumerate}
\item A universal ind-object is f-injective if and only if it is homogeneous.
\item Any two universal homogeneous ind-objects are isomorphic.
\item Every ind-object embeds into a universal homogeneous ind-object.
\end{enumerate}
\end{proposition}

We first prove a lemma that abstracts the classical back and forth argument.

\begin{lemma} \label{lem:back-forth}
Let $\Omega$ and $\Omega'$ be f-injective, let $\delta \colon X \to \Omega$ and $\delta' \colon X' \to \Omega'$ be embeddings with $X$ and $X'$ in $\cC$, and let $\gamma \colon X \to X'$ be an embedding. Then there exists an isomorphism $\alpha \colon \Omega \to \Omega'$ such that $\alpha \circ \delta = \delta' \circ \gamma$.
\end{lemma}

\begin{proof}
Let $\Omega=(\Omega_1 \to \Omega_2 \to \cdots)$. The map $\delta \colon X \to \Omega$ factors through some $\Omega_n$. Relabeling, we may as well assume that $\Omega_1=X$ and $\delta=\epsilon_{1,2}$ is the first transition map in $\Omega$. Similarly for $X' \to \Omega'$. Thus $\gamma$ is an embedding $\alpha_1 \colon \Omega_1 \to \Omega_1'$, and we want to extend $\alpha_1$ to an isomorphism $\alpha \colon \Omega \to \Omega'$.

Put $n(1)=m(1)=1$. Since $\Omega$ is f-injective, the embedding $\Omega_{n(1)} \to \Omega$ extends along $\alpha_1$. We can thus find an embedding $\beta_1 \colon \Omega'_{n(1)} \to \Omega_{n(2)}$ such that $\beta_1 \alpha_1=\epsilon_{n(1),n(2)}$. Since $\Omega'$ is f-injective, the embedding $\Omega'_{m(1)} \to \Omega'$ extends along $\beta_1$. We can thus find a map $\alpha_2 \colon \Omega_{n(2)} \to \Omega'_{m(2)}$ such that $\alpha_2 \beta_1 = \epsilon'_{m(1),m(2)}$. Continuing in this way, we obtain a commutative diagram
\begin{displaymath}
\xymatrix@C=4em{
\Omega_{n(1)} \ar[r]^{\epsilon} \ar[d]^{\alpha_1} & \Omega_{n(2)} \ar[r]^{\epsilon} \ar[d]^{\alpha_2} & \Omega_{n(3)} \ar[r]^{\epsilon} \ar[d]^{\alpha_3} & \cdots \\
\Omega'_{m(1)} \ar[r]^{\epsilon'} \ar[ru]^{\beta_1} & \Omega'_{m(2)} \ar[r]^{\epsilon'} \ar[ru]^{\beta_2} & \Omega'_{m(3)} \ar[r]^{\epsilon'} \ar[ru]^{\beta_3} & \cdots }
\end{displaymath}
Thus $\alpha \colon \Omega \to \Omega'$ extends $\alpha_1$, and is an isomorphism with inverse $\beta$.
\end{proof}

\begin{proof}[Proof of Proposition~\ref{prop:finj}]
(a) Suppose $\Omega$ is an ind-object that is universal and homogeneous. We show that $\Omega$ is f-injective. Let $\alpha \colon X \to Y$ and $\gamma \colon X \to \Omega$ be given. Since $\Omega$ is universal, there exists an embedding $\beta' \colon Y \to \Omega$. Since $\Omega$ is homogeneous, the two embeddings $\gamma$ and $\beta' \alpha$ of $X$ differ by an automorphism, that is, there is an automorphism $\sigma$ of $\Omega$ such that $\sigma \beta' \alpha=\gamma$. Taking $\beta=\sigma \beta'$ thus gives $\gamma=\beta \alpha$, and so $\Omega$ is f-injective.

The converse follows from Lemma~\ref{lem:back-forth}; in fact, this lemma shows that any f-injective object is homogeneous.

(b) Let $\Omega$ and $\Omega'$ be universal and f-injective. Since $\Omega'$ is universal, there is an embedding $\Omega_1 \to \Omega'$. Lemma~\ref{lem:back-forth} implies that this extends to an isomorphism $\Omega \to \Omega'$.

(c) Let $\Omega$ be a universal homogeneous ind-object and let $\Xi$ be another ind-object. Since $\Omega$ is universal, there is an embedding $\Xi_1 \to \Omega$. Since $\Omega$ is f-injective, this embedding extends to an embedding $\Xi_2 \to \Omega$, which in turn extends to an embedding $\Xi_3 \to \Omega$, and so on. In this way, we inductively construct an embedding $\Xi \to \Omega$.
\end{proof}

\subsection{Fra\"iss\'e categories}

Motivated by the importance of homogeneous objects, we introduce the following terminology:

\begin{definition}
We say that $\cC$ is a \defn{Fra\"iss\'e category} if it has a universal homogeneous ind-object.
\end{definition}

We prove one simply permanence property of these categories here. Recall from \S \ref{ss:coslice} the notion of the coslice category $\CS_X(\cC)$.

\begin{proposition} \label{A:coslice}
Let $\cC$ be a Fra\"iss\'e category and let $X$ be an object of $\cC$. Then the coslice category $\CS_X(\cC)$ is also Fra\"iss\'e.
\end{proposition}

\begin{proof}
Let $\Omega$ be a universal homogeneous ind-object of $\cC$. Since $\Omega$ is universal, there is an embedding $\alpha \colon X \to \Omega$. Reindexing, we assume that $\alpha=\alpha_1$ maps $X$ into $\Omega_1$. We let $\alpha_n \colon X \to \Omega_n$ be the composition of $\alpha_1$ with the transition map $\epsilon_{1,n} \colon \Omega_1 \to \Omega_n$. Thus each $(\Omega_n, \alpha_n)$ is an object of $\CS_X(\cC)$, and collectively they form an ind-object of $\CS_X(\cC)$, which we denote simply by $(\Omega, \alpha)$.

We claim that $(\Omega, \alpha)$ is a universal homogeneous ind-object of $\CS_X(\cC)$; this will complete the proof. Let $(Y, \beta)$ be an object of $\CS_X(\cC)$, where $\beta \colon X \to Y$ is a morphism in $\cC$. Since $\Omega$ is f-injective (as an ind-object of $\cC$), we can find a morphism $\gamma \colon Y \to \Omega$ such that $\gamma \circ \beta = \alpha$. Thus $\gamma$ defines a morphism $(Y,\beta) \to (\Omega,\alpha)$ in (the ind-category of) $\CS_X(\cC)$, which shows that $(\Omega, \alpha)$ is universal.

Finally, we show that $(\Omega, \alpha)$ is f-injective. Thus suppose we have maps $\phi \colon (Y, \beta) \to (\Omega, \alpha)$ and $\psi \colon (Y, \beta) \to (Z, \gamma)$. Unraveling all of this, we have a commutative diagram
\begin{displaymath}
\xymatrix@C=4em{
&& \Omega \\
X \ar@/^1.0pc/[rru]^{\alpha} \ar@/_1.0pc/[rrd]_{\gamma} \ar[r]^{\beta} & Y \ar[ru]^{\phi} \ar[rd]_{\psi} \\
&& Z \ar@{..>}[uu]_{\rho} }
\end{displaymath}
Since $\Omega$ is f-injective in $\cC$, we can find a morphism $\rho$ making the rightmost triangle (and thus the whole diagram) commute. It is clear that $\rho$ defines a morphism $(Z, \gamma) \to (\Omega, \alpha)$ such that $\rho \circ \psi = \phi$, which shows that $(\Omega, \alpha)$ is f-injective.
\end{proof}

\subsection{Fra\"iss\'e's theorem} \label{ss:fraisse}

We now prove a version of Fra\"iss\'es theorem, which characterizes Fra\"iss\'e categories. For this, we introduce two more conditions on $\cC$:
\begin{itemize}[leftmargin=1.5cm]
\item[(RCC)] \defn{Relative countable cofinality}: for any object $X$ of $\cC$ there exists a cofinal sequence of morphisms out of $X$, i.e., there is a sequence of morphisms $\{\alpha_n \colon X \to Y_n \}_{n \ge 1}$ such that if $\beta \colon X \to Y$ is any morphism then there is a morphism $\gamma \colon Y \to Y_n$ for some $n$ such that $\gamma \circ \beta = \alpha_n$.
\item[(AP)] \defn{Amalgamation property}: given embeddings $W \to X$ and $W \to Y$ there exists a commutative diagram
\begin{displaymath}
\xymatrix{
X \ar[r] & Z \\
W \ar[u] \ar[r] & Y \ar[u] }
\end{displaymath}
for some object $Z$.
\end{itemize}
If $\cC$ has an initial object then (CC) is a special case of (RCC), and (JEP) is a special case of (AP). We note that (RCC) holds for $\cC$ if and only if (CC) holds for $\CS_X(\cC)$ for all $X$; similarly, (AP) holds for $\cC$ if and only if (JEP) holds for $\CS_X(\cC)$ for all $X$.

The following is our categorical version of Fra\"iss\'e's theorem:

\begin{theorem} \label{A:fraisse}
The following are equivalent:
\begin{enumerate}
\item $\cC$ is a Fra\"iss\'e category, i.e., it has a universal homogeneous ind-object.
\item $\cC$ satisfies (CC), (RCC), (JEP), and (AP).
\item $\cC$ and each of its coslice categories $\CS_X(\cC)$ has a universal ind-object.
\end{enumerate}
\end{theorem}

Proposition~\ref{A:univ} shows that (b) and (c) are equivalent. If $\cC$ is Fra\"iss\'e then it obviously has a universal ind-object; since the coslice categories of $\cC$ are also Fra\"iss\'e (Proposition~\ref{A:coslice}), they too have universal ind-objects, and so (c) holds. In the remainder of \S \ref{ss:fraisse}, we assume that (b) holds and show that $\cC$ is Fra\"iss\'e.

\begin{lemma} \label{lem:fraisse3}
Suppose we have a diagram
\begin{displaymath}
\xymatrix@C=4em{
Y_1 & Y_2 &  & Y_{n-1} & Y_n \\
X_1 \ar[u]^{\beta_1} \ar[r]^{\alpha_1} & X_2 \ar[u]^{\beta_2} \ar[r]^{\alpha_2} & \cdots \ar[r]^{\alpha_{n-2}} & X_{n-1} \ar[u]^{\beta_{n-1}} \ar[r]^{\alpha_{n-1}} & X_n \ar[u]^{\beta_n} }
\end{displaymath}
in $\cC$. Then there exists an object $X_{n+1}$ and morphisms $\alpha_n \colon X_n \to X_{n+1}$ and $\gamma_i \colon Y_i \to X_{n+1}$ for $1 \le i \le n$ such that the diagram commutes, that is, we have
\begin{displaymath}
\gamma_i \circ \beta_i = \alpha_n \circ \cdots \circ \alpha_{i+1} \circ \alpha_i
\end{displaymath}
for all $1 \le i \le n$.
\end{lemma}

\begin{proof}
We proceed by induction on $n$. The statement is clear for $n=1$: we can take $X_2=Y_1$, with $\alpha_1=\beta_1$ and $\gamma_1=\id$. Suppose now the statement is true for $n-1$, and let us prove it for $n$. By the inductive hypothesis, we can find an object $X'_n$, an embedding $\alpha'_{n-1} \colon X_{n-1} \to X'_n$, and embeddings $\gamma'_i \colon Y_i \to X'_n$ for $1 \le i \le n-1$ such that the relevant diagram commutes. Consider the diagram
\begin{displaymath}
\xymatrix@C=4em{
X_{n-1} \ar[r]^{\alpha_{n-1}} \ar[d]_{\alpha'_{n-1}} & X_n \ar[r]^{\beta_n} & Y_n \ar@{..>}[d]^{\gamma_n} \\
X'_n \ar@{..>}[rr]^{\delta} && X_{n+1} }
\end{displaymath}
By (AP), we can find an object $X_{n+1}$ and morphisms $\delta$ and $\gamma_n$ making the diagram commute. We define $\alpha_n \colon X_n \to X_{n+1}$ to be the composition $\gamma_n \circ \beta_n$, and for $1 \le i \le n-1$ we define $\gamma_i \colon Y_i \to X_{n+1}$ to be the composition $\delta \circ \gamma'_i$. It is clear that the necessary conditions hold.
\end{proof}

Fix a cofinal sequence of objects $\{A(m)\}_{m \ge 1}$ as in (CC). For each object $X$, choose a cofinal sequence $\{\lambda_{X,m} \colon X \to X(m)\}_{m \ge 1}$ of morphisms out of $X$, as in (RCC). We assume that $A(m)$ embeds into $X(m)$ for each $m$; we can arrange this since (JEP) holds. A \defn{special ind-object} is an an ind-object $\Omega$ equipped with maps $\kappa_{n,m} \colon \Omega_n(m) \to \Omega_{n+m}$ for all $n,m \ge 1$ such that the diagram
\begin{displaymath}
\xymatrix{
\Omega_n \ar[rr]^{\epsilon_{n,n+m}} \ar[rd]_{\lambda_m} && \Omega_{n+m} \\
& \Omega_n(m) \ar[ru]_{\kappa_{n,m}} }
\end{displaymath}
commutes for all $n$ and $m$.

\begin{lemma} \label{lem:fraisse4}
A special ind-object exists.
\end{lemma}

\begin{proof}
We inductively construct $\Omega$. To start, we take $\Omega_1$ to be any object. Suppose now we have constructed $\Omega_1 \to \Omega_2 \to \cdots \to \Omega_{r-1}$ and the $\kappa_{n,m}$ for $n+m<r$. Consider the diagram
\begin{displaymath}
\xymatrix@C=4em{
\Omega_1(r-1) & \Omega_2(r-2) &  & \Omega_{r-2}(2) & \Omega_{r-1}(1) \\
\Omega_1 \ar[u]^{\lambda_{r-1}} \ar[r]^{\epsilon_1} & \Omega_2 \ar[u]^{\lambda_{r-2}} \ar[r]^{\epsilon_2} & \cdots \ar[r]^{\epsilon_{r-3}} & \Omega_{r-2} \ar[u]^{\lambda_2} \ar[r]^{\epsilon_{r-2}} & \Omega_{r-1} \ar[u]^{\lambda_1} }
\end{displaymath}
Here $\epsilon_i=\epsilon_{i,i+1}$. By Lemma~\ref{lem:fraisse3}, we can find an object $\Omega_r$ and morphisms $\epsilon_{r-1} \colon \Omega_{r-1} \to \Omega_r$ and $\kappa_{r,i} \colon \Omega_i(r-i) \to \Omega_r$ for $0 \le i \le r-1$ such that the diagram commutes. This constructs $\Omega$ up to level $r$, which completes the proof.
\end{proof}

\begin{lemma} \label{lem:fraisse5}
A special ind-object is universal and homogeneous.
\end{lemma}

\begin{proof}
Let $\Omega$ be special. We have embeddings $A(m) \to \Omega_1(m) \to \Omega$ for all $m \ge 1$. Since every object embeds into some $A(m)$, it follows that $\Omega$ is universal.

We now show that $\Omega$ is f-injective. Thus let $\alpha \colon X \to Y$ and $\gamma \colon X \to \Omega$ be given. Let $n$ be such that $X$ maps into $\Omega_n$. Consider a commutative diagram
\begin{displaymath}
\xymatrix{
\Omega_n \ar[r] & Z \\
X \ar[u]^{\gamma} \ar[r]^{\alpha} & Y \ar[u] }
\end{displaymath}
which exists by (AP). Now choose a map $Z \to \Omega_n(m)$ for some $m$ such that $\Omega_n \to Z \to \Omega_n(m)$ is $\lambda_m$; this exists by the definition of the $\lambda$'s. We thus have a commutative diagram
\begin{displaymath}
\xymatrix@C=3em{
\Omega_n \ar[r] \ar@/^1.5pc/[rr]^{\lambda_m} & Z \ar[r] & \Omega_n(m) \ar[r]^{\kappa_{n,m}} & \Omega_{n+m} \\
X \ar[u]^{\gamma} \ar[r]^{\alpha} & Y \ar[u] \ar@{..>}[rru]_{\beta} }
\end{displaymath}
where $\beta$ is defined to be the composition. Since the composition of the top line is the transition map $\epsilon_{n,n+m}$, it follows that $\gamma=\beta \circ \alpha$ as embeddings $X \to \Omega$. This completes the proof.
\end{proof}

Lemmas~\ref{lem:fraisse4} and~\ref{lem:fraisse5} show that a universal homogeneous object exists, which completes the proof of Theorem~\ref{A:fraisse}.

%
%
%

\subsection{Functors}

Suppose now that $\cD$ is a second category in which all morphisms are monic and $\Phi \colon \cC \to \cD$ is a functor. We now examine how $\Phi$ interacts with the classes of ind-objects considered above. We begin with universal objects. For this, we introduce the following property:
\begin{itemize}[leftmargin=1.5cm]
\item[(EP)] \defn{Embedding property}: for every object $Y$ of $\cD$ there is an object $Z$ of $\cC$ and an embedding $Y \to \Phi(Z)$.
\end{itemize}
If $\Phi$ is essentially surjective then (EP) holds, as one can then find an isomorphism $Y \to \Phi(Z)$.

\begin{proposition} \label{A:Phi-univ}
Suppose that $\cC$ has a universal ind-object. Then the following are equivalent:
\begin{enumerate}
\item $\Phi$ satisfies (EP).
\item $\Phi$ maps some universal ind-object of $\cC$ to a universal ind-object of $\cD$.
\item $\Phi$ maps every universal ind-object of $\cC$ to a universal ind-object of $\cD$.
\end{enumerate}
\end{proposition}

\begin{proof}
(a) $\implies$ (c). Let $\Omega$ be a universal ind-object of $\cC$. We must show that $\Phi(\Omega)$ is universal. Let $Y$ be an object of $\cD$. By (EP), there is an embedding $Y \to \Phi(Z)$ for some object $Z$ of $\cC$. Since $\Omega$ is universal, there is an embedding $Z \to \Omega$, which yields an embedding $\Phi(Z) \to \Phi(\Omega)$. We thus obtain an embedding $Y \to \Phi(\Omega)$, and so $\Phi(\Omega)$ is universal.

(c) $\implies$ (b) is trivial.

(b) $\implies$ (a). Let $\Omega$ be a universal ind-object of $\cC$ such that $\Phi(\Omega)$ is universal. Let $Y$ be an object of $\cD$. Then there is an embedding $Y \to \Phi(\Omega)$. This comes from an embedding $Y \to \Phi(\Omega_n)$ for some $n$, and so (EP) holds.
\end{proof}

We now examine how $\Phi$ interacts with homogeneous objects. For this, we introduce a variant of the (EP) property:
\begin{itemize}[leftmargin=1.5cm]
\item[(REP)] \defn{Relative embedding property}: given a morphism $\alpha \colon \Phi(X) \to Y$ in $\cD$, there exists a commutative diagram
\begin{displaymath}
\xymatrix{
\Phi(X) \ar[rd]_{\alpha} \ar[rr]^{\Phi(\gamma)} && \Phi(Z) \\
& Y \ar[ru]_{\beta} }
\end{displaymath}
where $\gamma \colon X \to Z$ is a morphism in $\cC$.
\end{itemize}
Note that (REP) holds if and only if for every object $X$ of $\cC$ the natural functor $\Phi \colon \CS_X(\cC) \to \CS_{\Phi(X)}(\cD)$ satisfies (EP).

\begin{proposition} \label{A:Phi-homo}
Suppose that $\cC$ and $\cD$ are Fra\"iss\'e categories. Then the following are equivalent:
\begin{enumerate}
\item $\Phi$ maps any universal homogeneous ind-object of $\cC$ to a universal homogeneous ind-object of $\cD$.
\item $\Phi$ satisfies (EP) and (REP)
\item $\Phi$ and the induced functors $\CS_X(\cC) \to \CS_{\Phi(X)}(\cD)$ (for any object $X$ of $\cC$) map universal ind-objects to universal ind-objects.
\end{enumerate}
\end{proposition}

We note that because all universal homogeneous ind-objects are isomorphic, condition (a) will be satisfied if $\Phi$ maps any particular universal homogeneous ind-object of $\cC$ to a universal homogeneous ind-object of $\cD$.

\begin{proof}
Proposition~\ref{A:Phi-univ} shows that (b) and (c) are equivalent.

Suppose that (a) holds. Let $\Omega$ be a universal homogeneous ind-object of $\cC$, so that $\Phi(\Omega)$ is a universal homogeneous ind-object of $\cD$. Given an embedding $\alpha \colon X \to \Omega$, (the proof of) Proposition~ \ref{A:coslice} shows that $(\Omega, \alpha)$ is a universal homogeneous ind-object of $\Delta_X(\cC)$. By the same reasoning, $(\Phi(\Omega), \Phi(\alpha))=\Phi(\Omega, \alpha)$ is a universal homogeneous object of $\Delta_{\Phi(X)}(\cD)$. Thus (c) holds.


Suppose (b) holds, let $\Omega$ be a universal homogeneous ind-object of $\cC$, and let $\Omega'$ be a universal homogeneous object of $\cD$. Put $n(1)=1$ and choose an embedding $\alpha_1 \colon \Phi(\Omega_{n(1)}) \to \Omega'_{m(1)}$ for some $m(1)$, which is possible since $\Omega'$ is universal. By (REP), $\alpha_1$ factors into a morphism $\Phi(\Omega_{n(1)} \to Z)$ for some $Z$ in $\cC$. Since $\Omega$ is f-injective, the morphism $\Omega_{n(1)} \to Z$ factors into one of the transition maps $\epsilon_{n(1),n(2)} \colon \Omega_{n(1)} \to \Omega_{n(2)}$ for some $n(2)$. The upshot is that we can find a commutative diagram
\begin{displaymath}
\xymatrix{
\Phi(\Omega_{n(1)}) \ar[rd]_{\alpha_1} \ar[rr]^{\Phi(\epsilon_{n(1),n(2)})} && \Phi(\Omega_{n(2)}) \\
& \Omega'_{m(1)} \ar[ru]_{\beta_1} }
\end{displaymath}
for some $\beta_1$. Now, since $\Omega'$ is f-injective, we can extend the embedding $\Omega'_{m(1)} \to \Omega'$ along $\beta_1$. We can thus extend the above diagram to a commutative diagram
\begin{displaymath}
\xymatrix{
\Phi(\Omega_{n(1)}) \ar[rd]_{\alpha_1} \ar[rr]^{\Phi(\epsilon_{n(1),n(2)})} && \Phi(\Omega_{n(2)}) \ar[rd]^{\alpha_2} \\
& \Omega'_{m(1)} \ar[ru]_{\beta_1} \ar[rr]_{\epsilon'_{m(1),m(2)}} && \Omega'_{m(2)}}
\end{displaymath}
for some $m(2)$. Continuing in this manner, we construct an isomorphism $\alpha \colon \Phi(\Omega) \to \Omega'$. Thus $\Phi(\Omega)$ is a universal homogeneous ind-object of $\cD$, and so (a) holds.
\end{proof}

%
%
%


\begin{thebibliography}{BDDE}

\bibitem[AH]{AH} Tigran Ananyan, Melvin Hochster. Small subalgebras of polynomial rings and Stillman’s conjecture. \textit{J.\ Amer.\ Math.\ Soc.} \textbf{33} (2020), pp.~291–309. \DOI{10.1090/jams/932} \arxiv{1610.09268}


\bibitem[BDDE]{BDDE} Arthur Bik, Alessandro Danelon, Jan Draisma, Rob H. Eggermont. Universality of high-strength tensors. \textit{Vietnam J.\ Math.} \textbf{50} (2022), pp.~557--580. \DOI{10.1007/s10013-021-00522-7} \arxiv{2105.00016}

\bibitem[BDES]{polygeom} Arthur Bik, Jan Draisma, Rob H. Eggermont, Andrew Snowden. The geometry of polynomial representations. \textit{Int.\ Math.\ Res.\ Not.\ IMRN} (2022). \DOI{10.1093/imrn/rnac220} \arxiv{2105.12621}

\bibitem[BDS]{isocubic} Arthur Bik, Alessandro Danelon, Andrew Snowden. Isogeny classes of cubic spaces. \arxiv{2207.13951}

\bibitem[Cam]{Cameron} Peter J.\ Cameron. Oligomorphic permutation groups. London Mathematical Society Lecture Note Series, vol. 152, Cambridge University Press, Cambridge, 1990.

\bibitem[Car]{Caramello} Olivia Caramello. Fraisse's construction from a topos-theoretic perspective. \textit{Log.\ Univers.} \textbf{8} (2014), no.~2, 261--281. \DOI{10.1007/s11787-014-0104-6} \arxiv{0805.2778}


\bibitem[Dra]{Draisma} Jan Draisma. Topological Noetherianity of polynomial functors. \textit{J.\ Amer.\ Math.\ Soc.} \textbf{32} (2019), no.~3, pp.~691--707. \DOI{10.1090/jams/923} \arxiv{1705.01419}

\bibitem[Del]{Deligne} P. Deligne. La cat\'egorie des repr\'esentations du groupe sym\'etrique $S_t$, lorsque $t$ n’est pas un entier naturel. In: Algebraic Groups and Homogeneous Spaces, in: Tata Inst.\ Fund.\ Res.\ Stud.\ Math., Tata Inst.\ Fund.\ Res., Mumbai, 2007, pp.~209--273. \\
Available at: {\tiny\url{https://www.math.ias.edu/files/deligne/Symetrique.pdf}}

\bibitem[DES]{des} Harm Derksen, Rob H.~Eggermont, Andrew Snowden. Topological noetherianity for cubic polynomials. \textit{Algebra Number Theory} \textbf{11}(9) (2017), pp.\ 2197--2212. \DOI{10.2140/ant.2017.11.2197} \arxiv{1701.01849}

\bibitem[DG1]{DrosteGobel1} Manfred Droste, R\"udier G\"obel. A categorial theorem on universal objects and its application in abelian group theory and computer science. \textit{Contemp.\ Math.} \textbf{131} (Part 3) 1992, pp.~49--74. \DOI{10.1090/conm/131.3}

\bibitem[DG2]{DrosteGobel2} Manfred Droste, R\"udier G\"obel. Universal domains and the amalgamation property. \textit{Math.\ Structures Comput.\ Sci.} \textbf{3} (1993), no.~2, pp.~137--159. \DOI{10.1017/S0960129500000177}

\bibitem[DM]{DeligneMilne} P.~Deligne, J.~Milne. Tannakian Categories. In ``Hodge cycles, motives, and Shimura varieties,'' Lecture Notes in Math., vol.\ 900, Springer--Verlag, 1982. \href{http://doi.org/10.1007/978-3-540-38955-2_4}{\color{purple}{\tiny\tt DOI:10.1007/978-3-540-38955-2\_4}}\\
Available at: {\tiny\url{http://www.jmilne.org/math/xnotes/tc.html}}

\bibitem[DPS]{koszulcategory} Elizabeth Dan-Cohen, Ivan Penkov, Vera Serganova. A Koszul category of representations of finitary Lie algebras. \textit{Adv.\ Math.} {\bf 289} (2016), 250--278. \DOI{10.1016/j.aim.2015.10.023} \arxiv{1105.3407v2}


\bibitem[ESS]{ess} Daniel Erman, Steven V Sam, Andrew Snowden. Big polynomial rings and Stillman's conjecture. \textit{Invent.\ Math.} \textbf{218} (2019), no.~2, pp.~413--439. \DOI{10.1007/s00222-019-00889-y} \arxiv{1801.09852}


\bibitem[Fra]{Fraisse} R. Fra\"iss\'e. Sur certaines relations qui g\'en\'eralisent l’order des nombres rationnels. \textit{C.~R.~Acad.~Sci.} \textbf{237} (1953), pp.~540--542.

\bibitem[Gra]{Granger} Nicolas Granger. Stability, simplicity, and the model theory of bilinear forms. PhD
thesis, University of Manchester, Manchester, 1999.

\bibitem[Gro]{Gross} Herbert Gross. Quadratic forms in infinite dimensional vector spaces. Progr.\ Math. \textbf{1}, Birkh\"auser, Boston, 1979.

\bibitem[Hod]{Hodges} Wilfrid Hodges. Model theory. Encyclopedia of Mathematics and its Applications, 42 Cambridge University Press, Cambridge, 1993.

\bibitem[HS1]{repst} Nate Harman, Andrew Snowden. Oligomorphic groups and tensor categories. \arxiv{2204.04526}

\bibitem[HS2]{discrete} Nate Harman, Andrew Snowden. Discrete pre-Tannakian categories. \arxiv{2304.05375}

\bibitem[HS3]{homoten2} Nate Harman, Andrew Snowden. Ultrahomogeneous tensor structures. II. In preparation.

\bibitem[Irw]{Irwin} Trevor L. Irwin. Fra\"iss\'e limits and colimits with applications to continua. Ph.\ D.~Thesis, Indiana University, 2007.

\bibitem[Kas]{Kamsma} Mark Kamsma. Bilinear spaces over a fixed field are simple unstable. \textit{Ann.\ Pure Appl.\ Logic} \textbf{174} (2023), no.~6. \DOI{10.1016/j.apal.2023.103268} \arxiv{2203.04844}

\bibitem[KaZ1]{KaZ1} David Kazhdan, Tamar Ziegler. On ranks of polynomial. \textit{Algebr.\ Represent.\ Theory} \textbf{21} (2018), no.~5, pp.~1017--1021. \DOI{10.1007/s10468-018-9783-7} \arxiv{1802.04984}

\bibitem[KaZ2]{KaZ2} David Kazhdan, Tamar Ziegler. Properties of high rank subvarieties of affine spaces. \textit{Geom.\ Funct.\ Anal.} \textbf{30} (2020), pp.~1063--1096. \DOI{10.1007/s00039-020-00542-4} \arxiv{1902.00767}

\bibitem[Kub]{Kubis} Wies\l{}aw Kubi\'s. Fra\"iss\'e sequences: category-theoretic approach to universal homogeneous structures. \arxiv{0711.1683}

\bibitem[LW]{LachlanWoodrow} A.~H.~Lachlan, Robert E.~Woodrow. Countable ultrahomogeneous undirected graphs. \textit{Trans.\ Amer.\ Math.\ Soc.} \textbf{262} (1980), pp.~51--94. \DOI{10.2307/1999974}

\bibitem[Mac]{Macpherson} Dugald Macpherson. A survey of homogeneous structures. \textit{Discrete Math.} \textbf{311} (2011), no.~15, pp.~1599--1634. \DOI{10.1016/j.disc.2011.01.024}


\bibitem[PP]{PechPech} Christian Pech, Maja Pech. Universal homogeneous constraint structures and the hom-equivalence classes of weakly oligomorphic structures.
\arxiv{1203.6086}

\bibitem[PSe]{penkovserganova} Ivan Penkov, Vera Serganova. Categories of integrable $sl(\infty)$-, $o(\infty)$-, $sp(\infty)$-modules. \textit{Representation Theory and Mathematical Physics}, Contemp.\ Math.\ {\bf 557}, AMS 2011, pp. 335--357. \DOI{10.1090/conm/557} \arxiv{1006.2749v1}

\bibitem[PSt]{penkovstyrkas} Ivan Penkov, Konstantin Styrkas. Tensor representations of classical locally finite Lie algebras. \textit{Developments and trends in infinite-dimensional Lie theory}, Progr. Math. {\bf 288}, Birkh\"auser Boston, Inc., Boston, MA, 2011, pp. 127--150. \arxiv{0709.1525v1}

\bibitem[Sno]{tcares} Andrew Snowden. Stable representation theory: beyond the classical groups. \arxiv{2109.11702}

\bibitem[SS1]{expos} Steven~V Sam, Andrew Snowden. Introduction to twisted commutative algebras. \arxiv{1209.5122}

\bibitem[SS2]{infrank} Steven~V Sam, Andrew Snowden. Stability patterns in representation theory. \textit{Forum Math. Sigma} {\bf 3} (2015), e11, 108 pp. \DOI{10.1017/fms.2015.10} \arxiv{1302.5859v2}

\bibitem[SS3]{spinrep} Steven V Sam, Andrew Snowden. Infinite rank spinor and oscillator representations. \textit{J.\ Comb.\ Algebra} (2017), no.~2, pp.~145--183. \DOI{10.4171/JCA/1-2-2} \arxiv{1604.06368}


\end{thebibliography}
\end{document}